\documentclass[reqno, 12pt]{amsart}
\pdfoutput=1
\makeatletter
\let\origsection=\section \def\section{\@ifstar{\origsection*}{\mysection}} 
\def\mysection{\@startsection{section}{1}\z@{.7\linespacing\@plus\linespacing}{.5\linespacing}{\normalfont\scshape\centering\S}}
\makeatother        

\usepackage{amsmath,amssymb,amsthm}
\usepackage{mathrsfs}
\usepackage{dsfont}

\usepackage{todonotes}
\usepackage{verbatim}
 
\usepackage{graphicx}
\usepackage{xcolor}
\usepackage{tikz}
\usetikzlibrary{arrows}

\usepackage[backref]{hyperref}
\hypersetup{
    colorlinks,
    linkcolor={red!60!black},
    citecolor={green!60!black},
    urlcolor={blue!60!black},
    pdftitle={Extremal problems for graphs with small local crossing number},
    pdfauthor={J. Pascal Gollin, Kevin Hendrey, Abhishek Methuku, Casey Tompkins, and Xin Zhang}
}

\usepackage[open,openlevel=2,atend]{bookmark}

\usepackage[abbrev,msc-links,backrefs]{amsrefs} 
\usepackage{doi}

\renewcommand{\PrintDOI}[1]{\doi{#1}}

\usepackage[T1]{fontenc}
\usepackage{lmodern}
\usepackage[babel]{microtype}
\usepackage[english]{babel}

\usepackage{geometry}
\numberwithin{equation}{section}
\numberwithin{figure}{section}

\usepackage{enumitem}

\let\polishlcross=\l
\def\l{\ifmmode\ell\else\polishlcross\fi}

\def\paragraph#1{%
  \noindent\textbf{#1.}\enspace}

\let\emptyset=\varnothing

\makeatletter
\def\moverlay{\mathpalette\mov@rlay}
\def\mov@rlay#1#2{\leavevmode\vtop{   \baselineskip\z@skip \lineskiplimit-\maxdimen
   \ialign{\hfil$\m@th#1##$\hfil\cr#2\crcr}}}
\newcommand{\charfusion}[3][\mathord]{
    #1{\ifx#1\mathop\vphantom{#2}\fi
        \mathpalette\mov@rlay{#2\cr#3}
      }
    \ifx#1\mathop\expandafter\displaylimits\fi}
\makeatother

\DeclareFontFamily{U}  {MnSymbolC}{}
\DeclareSymbolFont{MnSyC}         {U}  {MnSymbolC}{m}{n}
\DeclareFontShape{U}{MnSymbolC}{m}{n}{
    <-6>  MnSymbolC5
   <6-7>  MnSymbolC6
   <7-8>  MnSymbolC7
   <8-9>  MnSymbolC8
   <9-10> MnSymbolC9
  <10-12> MnSymbolC10
  <12->   MnSymbolC12}{}
\DeclareMathSymbol{\powerset}{\mathord}{MnSyC}{180}

\let\epsilon=\varepsilon

\let\rho=\varrho
\let\theta=\vartheta
\let\phi=\varphi

\usepackage{thmtools, thm-restate}
\usepackage{setspace}
\usepackage{multirow}

\theoremstyle{plain}
\newtheorem{thm}{Theorem}[section]
\newtheorem{theorem}[thm]{Theorem}
\newtheorem{lemma}[thm]{Lemma}
\newtheorem{corollary}[thm]{Corollary}
\newtheorem{proposition}[thm]{Proposition}

\newtheorem{conjecture}[thm]{Conjecture}
\newtheorem{question}[thm]{Question}
\newtheorem*{claim*}{Claim}

\newtheorem{thm-intro}{Theorem}[]
\newtheorem{conj-intro}[thm-intro]{Conjecture}

\newtheorem{question-intro}[thm-intro]{Question}
\newtheorem*{meta-question*}{Meta question}

\theoremstyle{definition}
\newtheorem{definition}[thm]{Definition}
\newtheorem{remark}[thm]{Remark}

\newtheorem*{example*}{Example}

\newcommand{\abs}[1]{\ensuremath{{\lvert {#1} \rvert}}}
\newcommand{\restricted}{\ensuremath{{\upharpoonright}}}

\newcommand{\floor}[1]{\left\lfloor{#1}\right\rfloor}
\newcommand{\ceil}[1]{\left\lceil{#1}\right\rceil}

\begin{document}

\author[J.~P.~Gollin]{J.~Pascal Gollin}
\address{J. Pascal Gollin, Discrete Mathematics Group, Institute for Basic Science (IBS), 55 Expo-ro, Yuseong-gu, Daejeon, Korea, 34126}
\email{\tt pascalgollin@ibs.re.kr}
\thanks{The first author was supported by the Institute for Basic Science (IBS-R029-Y3).}

\author[K.~Hendrey]{Kevin Hendrey}
\address{Kevin Hendrey, Discrete Mathematics Group, Institute for Basic Science (IBS), 55 Expo-ro, Yuseong-gu, Daejeon, Korea, 34126}
\email{\tt kevinhendrey@ibs.re.kr}
\thanks{The second, third, and fourth authors were supported by the Institute for Basic Science (IBS-R029-C1).}

\author[A.~Methuku]{Abhishek Methuku}
\address{Abhishek Methuku, School of Mathematics, University of Birmingham, Edgbaston, Birmingham, B15 2TT, UK}
\email{\tt abhishekmethuku@gmail.com}
\thanks{Additionally the third author was supported by the EPSRC grant no EP/S00100X/1 (A. Methuku).}

\author[C.~Tompkins]{Casey Tompkins}
\address{Casey Tompkins, Discrete Mathematics Group, Institute for Basic Science (IBS), 55 Expo-ro, Yuseong-gu, Daejeon, Korea, 34126}
\email{\tt caseytompkins@ibs.re.kr}

\author[X.~Zhang]{Xin Zhang}
\address{Xin Zhang, School of Mathematics and Statistics, Xidian University, 266 Xifeng Rd., Xi'an, Shaanxi, China, 710126}
\email{\tt xzhang@xidian.edu.cn}
\thanks{The last author was supported by the NSFC grant No.\,11871055.}

\title[Counting cliques in $1$-planar graphs]{Counting cliques in $1$-planar graphs}

\date\today



\begin{abstract}
    The problem of maximising the number of cliques among $n$-vertex graphs from various graph classes has received considerable attention. 
    We investigate this problem for the class of $1$-planar graphs where we determine precisely the maximum total number of cliques as well as the maximum number of cliques of any fixed size.
    We also precisely characterise the extremal graphs for these problems. 
\end{abstract}

\maketitle

\section{Introduction}
\label{sec:intro}

A $1$-planar graph is a graph which can be drawn in the plane such that each edge is crossed by at most one other edge (see Section~\ref{sec:prelims} for a formal definition). 
This class forms a natural extension of the class of planar graphs, and was first introduced by Ringel in 1965~\cite{ringel1965}. 
His motivation was the problem of simultaneously colouring vertices and faces of a plane graph so that each face and vertex gets a distinct colour from all neighbouring faces and vertices. 
Such a colouring corresponds to a vertex colouring of a $1$-planar graph, and 
Ringel made progress on this problem by showing that every $1$-planar graph is $7$-colourable. 
The problem was later fully answered by Borodin~\cite{Borodin84}, who showed that six colours are sufficient to colour every $1$-planar graph (the triangular prism has no such $5$-vertex-face-colouring). 

Since its introduction, the class of~$1$-planar graphs has proven to be a topic of much interest, as reflected by the large body of research in the area (some results on the topic are surveyed in~\cites{DidimoLM19,HongTokuyama20,KobourovLM17}). 
One reason for interest in the class is that it is closely related to the class of planar graphs, while having a number of important qualitative differences. 
Most notably, whereas the classical results of Kuratowski~\cite{Kuratowski} and Wagner~\cite{MR1513158} tell us that planar graphs can be defined as exactly those graphs with no~$K_{3,3}$ or~$K_5$ (topological) minor, the class of $1$-planar graphs has no such characterisation. 
In fact it is easy to see that every graph can be made $1$-planar by subdividing its edges sufficiently many times. 
Another key distinction is that planar graphs can be recognised in time linear in the number of vertices~\cite{HopcroftT74}, but the recognition problem for $1$-planar graphs is NP-complete~\cites{BodlaenderG07, KorzhikMohar13}. 
On the other hand, $1$-planar graphs are only a linear number of edges away from being planar. 
In fact, whenever an $n$-vertex graph is drawn in the plane so that each edge is crossed at most once and adjacent edges do not cross, the total number of crossings is at most~${n-2}$~\cite{ZLW12}.

In extremal graph theory, determining the maximum number of edges in an $n$-vertex graph from a given graph class is a fundamental problem. 
The classical theorem of Tur\'{a}n~\cite{turan1941external} asserts that in the class of~$K_r$-free graphs, the maximum number of edges is attained (uniquely) by the complete multipartite graph with~${r-1}$ parts each of size~${\floor{\frac{n}{r-1}}}$ or~${\ceil{\frac{n}{r-1}}}$.
In the class of planar graphs it is a simple consequence of Euler's formula that an $n$-vertex graph contains at most~${3n-6}$ edges.
The bound on the number of crossings in $1$-planar drawings implies that $n$-vertex $1$-planar graphs have at most ${4n-8}$ edges, although this was first shown by Bodendiek, Schumacher and Wagner \cite{BodendiekSW83}.
In the same paper, they characterised the $n$-vertex $1$-planar graphs which achieve this bound as exactly the graphs obtained from $3$-connected planar quadrangulations by adding a pair of crossing edges to every face.

Going beyond maximising edges it is natural to consider maximising the number of cliques of a given size~$t$ among $n$-vertex graphs from a class of graphs~$\mathcal{G}$. 
Zykov~\cite{Zykov} and Erd\H{o}s~\cite{erdos} independently determined this number for the class of $K_r$-free graphs for all positive integers~$t$ and~$n$. 
Hakimi and Schmeichel~\cite{MR519175} determined the maximum number of triangles in an $n$-vertex planar graph (see Theorem~\ref{thm:planarK3}), and the maximum number of cliques of size four in an $n$-vertex planar graph was determined independently by Alon and Caro~\cite{MR791009}, and Wood~\cite{wood}. 
We determine the maximum number of cliques of any fixed size in an $n$-vertex $1$-planar graph. 
Note that for ${n \leq 6}$, the clique~$K_n$ is $1$-planar, and contains the maximum possible number of cliques of any fixed size. 
For ${t \geq 7}$, no $1$-planar graph contains a clique of size~$t$. 
All remaining cases are covered by the following theorems. 

\begin{thm}
    \label{thm:mainintro1}
    Given integers~${k \geq 2}$ and ${s \in \{0,1,2\}}$, the maximum number~${f_3(3k+s)}$ of subgraphs isomorphic to~$K_3$ in a $1$-planar graph with~${3k + s}$ vertices is given by 
    \[
        f_3(3k+s) = 
        \begin{cases}
            32          & \textnormal{ if } 3k+s =8,\\
            19k+5s-18   & \textnormal{ otherwise. }
        \end{cases}
    \]
\end{thm}

\begin{thm}
    \label{thm:mainintro2}
    Given integers~${k \geq 2}$, and ${s \in \{0,1,2\}}$ and~${t \in \{4,5,6\}}$, the maximum number ${f_t(3k+s)}$
    of subgraphs isomorphic to~$K_t$ in a $1$-planar graph with~${3k + s}$ vertices is given by 
    \[
        f_t(3k+s) = (k-1)\binom{6}{t}+\binom{s+3}{t}.\]
\end{thm}

Building on these questions, it is natural to ask for a structural characterisation of the extremal graphs. 
The planar graphs which maximise the number of triangles are the planar graphs which can be formed from~$K_3$ by iteratively pasting copies of~$K_4$ on facial triangles. 
These graphs are called Apollonian networks, 
and they are also the planar graphs with the maximum number of cliques of size four. 
We provide such a structural characterisation for $1$-planar graphs, determining precisely which graphs attain the bounds in Theorems~\ref{thm:mainintro1} and~\ref{thm:mainintro2}. 
When the number of vertices is divisible by three, these extremal graphs are analogous to Apollonian networks; they are the $1$-planar graphs formed from $K_3$ by iteratively pasting copies of~$K_6$ on facial triangles\footnotemark\ (see Figure~\ref{fig:K6-stichtings}). 
We give a formal description of all extremal graphs in Section~\ref{sec:stitchings}. 

\footnotetext{We will see in Section~\ref{sec:treedec} that these graphs are exactly the graphs which are isomorphic to the strong product~${K_3 \boxtimes P_{m}}$ of a triangle~$K_3$ and a path~$P_{m}$ of length~$m-1$ for some positive integer~$m$.}

    \begin{figure}[htbp]
        \centering
        \begin{tikzpicture}
            [scale=0.8]
            \tikzset{vertex/.style = {circle, draw, fill=black!50, inner sep=0pt, minimum width=4pt}}
            \tikzset{edge0/.style = {line width=1.2pt, black}}
            \tikzset{edge1/.style = {line width=1pt, blue, opacity=0.7}}
            
            \node [vertex] (v0) at (210:3) {};
            \node [vertex] (v1) at (330:3) {};
            \node [vertex] (v2) at (90:3) {};
            \node [vertex] (w0) at (210:1.5) {};
            \node [vertex] (w1) at (330:1.5) {};
            \node [vertex] (w2) at (90:1.5) {};
            
            \foreach \i in {0,1,2} {
                \pgfmathtruncatemacro{\j}{Mod(\i+1,3)}
                \foreach \v in {v,w} {
                    \draw [edge0] (\v\i) edge (\v\j) {};
                };
                \draw [edge0] (v\i) edge (w\i) {};
                \draw [edge1] (v\i) edge (w\j) {};
                \draw [edge1] (v\j) edge (w\i) {};
            };
        \end{tikzpicture}
        \begin{tikzpicture}
            [scale=0.8]
            \tikzset{vertex/.style = {circle, draw, fill=black!50, inner sep=0pt, minimum width=4pt}}
            \tikzset{edge0/.style = {line width=1.2pt, black}}
            \tikzset{edge1/.style = {line width=1pt, blue, opacity=0.7}}
            
            \node [vertex] (v0) at (210:3) {};
            \node [vertex] (v1) at (330:3) {};
            \node [vertex] (v2) at (90:3) {};
            \node [vertex] (w0) at (210:1.875) {};
            \node [vertex] (w1) at (330:1.875) {};
            \node [vertex] (w2) at (90:1.875) {};
            \node [vertex] (x0) at (210:0.75) {};
            \node [vertex] (x1) at (330:0.75) {};
            \node [vertex] (x2) at (90:0.75) {};
            
            \foreach \i in {0,1,2} {
                \pgfmathtruncatemacro{\j}{Mod(\i+1,3)}
                \foreach \v in {v,w,x} {
                    \draw [edge0] (\v\i) edge (\v\j) {};
                };
                \draw [edge0] (v\i) edge (w\i) {};
                \draw [edge0] (w\i) edge (x\i) {};
                \draw [edge1] (v\i) edge (w\j) {};
                \draw [edge1] (v\j) edge (w\i) {};
                \draw [edge1] (w\i) edge (x\j) {};
                \draw [edge1] (w\j) edge (x\i) {};
            };
        \end{tikzpicture}
        \begin{tikzpicture}
            [scale=0.8]
            \tikzset{vertex/.style = {circle, draw, fill=black!50, inner sep=0pt, minimum width=4pt}}
            \tikzset{edge0/.style = {line width=1.2pt, black}}
            \tikzset{edge1/.style = {line width=1pt, blue, opacity=0.7}}
            
            \node [vertex] (v0) at (210:3) {};
            \node [vertex] (v1) at (330:3) {};
            \node [vertex] (v2) at (90:3) {};
            \node [vertex] (w0) at (210:2.25) {};
            \node [vertex] (w1) at (330:2.25) {};
            \node [vertex] (w2) at (90:2.25) {};
            \node [vertex] (x0) at (210:1.5) {};
            \node [vertex] (x1) at (330:1.5) {};
            \node [vertex] (x2) at (90:1.5) {};
            \node [vertex] (y0) at (210:0.75) {};
            \node [vertex] (y1) at (330:0.75) {};
            \node [vertex] (y2) at (90:0.75) {};
            
            \foreach \i in {0,1,2} {
                \pgfmathtruncatemacro{\j}{Mod(\i+1,3)}
                \foreach \v in {v,w,x,y} {
                    \draw [edge0] (\v\i) edge (\v\j) {};
                };
                \draw [edge0] (v\i) edge (w\i) {};
                \draw [edge0] (w\i) edge (x\i) {};
                \draw [edge0] (x\i) edge (y\i) {};
                \draw [edge1] (v\i) edge (w\j) {};
                \draw [edge1] (v\j) edge (w\i) {};
                \draw [edge1] (w\i) edge (x\j) {};
                \draw [edge1] (w\j) edge (x\i) {};
                \draw [edge1] (x\i) edge (y\j) {};
                \draw [edge1] (x\j) edge (y\i) {};
            };
        \end{tikzpicture}
        \caption{Examples of extremal graphs for Theorems~\ref{thm:mainintro1},~\ref{thm:mainintro2} and~\ref{thm:mainintro3}.}
        \label{fig:K6-stichtings}
    \end{figure}
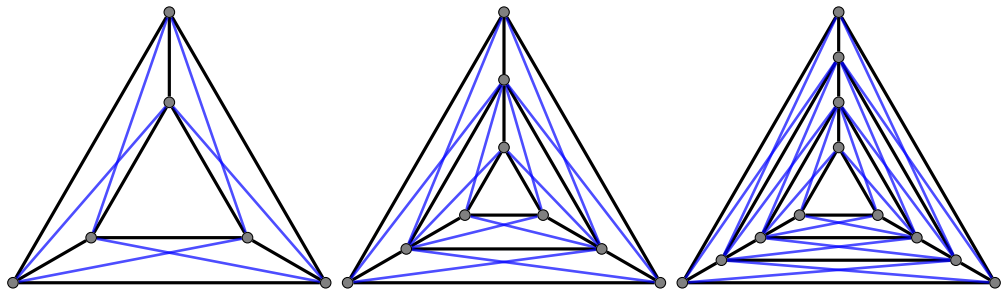

In the class of planar graphs, the $n$-vertex graphs with the maximum number of cliques of size four also have the maximum number of triangles and the maximum number of edges \cite{Zykov}, and therefore have the maximum total number of cliques. 
In fact this was Wood's motivation for counting cliques of size four in planar graphs \cite{wood}. 
The same is true of the class of $K_r$-free graphs, where the graphs with the maximum number of cliques of size~$t$ are exactly those with the maximum number of edges for all~$t$ between~$2$ and~$r$. 
In the case of $1$-planar graphs, the situation is not as simple. 
In fact, for~$n$ at least~$9$ and~$t$ between~$3$ and~$6$, every $n$-vertex $1$-planar graph which maximises the number of cliques of size~$t$ has strictly fewer than~${4n-8}$ edges (i.e.\ the maximum number of edges in an $n$-vertex $1$-planar graph). 
Furthermore, when~$n$ is at least~$7$ and not divisible by~$3$, there is no $n$-vertex $1$-planar graph which simultaneously maximises the number of triangles and cliques of size five. 
We nevertheless determine the maximum total number of cliques in $n$-vertex $1$-planar graphs. 
We also characterise the $1$-planar graphs attaining this bound, which we describe in Section~\ref{sec:stitchings}. 
When~$n$ is divisible by three, these are the previously described graphs which maximise the number of cliques of all fixed sizes between~$3$ and~$6$ (see Figure~\ref{fig:K6-stichtings}).

\begin{thm}\label{thm:mainintro3}
    Given integers~${k \geq 1}$ and~${s \in \{0,1,2\}}$, the maximum number~${f(3k+s)}$ of cliques in a $1$-planar graph with~${3k + s}$ vertices is given by 
    \[
        f(3k+s) = 
        56(k-1)+2^{s+3}.
    \]
\end{thm}

We summarise our results in the context of the extremal functions for counting cliques in $1$-planar graphs in Table~\ref{tab:results}. 

\begin{table}[htbp]\label{tab:results}
    \centering
    \begin{spacing}{1.5}
    \begin{tabular}{|c||c||c|c||c|c||c|c|}
        \hline
        & 
        \multirow{3}{*}{$n \leq 6$} & 
        \multicolumn{6}{c|}{$n \geq 7$}\\
        
        \cline{3-8}
        & & 
        \multicolumn{2}{c||}{$n\equiv 1$ (mod 3)} & 
        \multicolumn{2}{c||}{$n\equiv 2$ (mod 3)} &
        \multicolumn{2}{c|}{$n\equiv 0$ (mod 3)} \\
        
        \cline{3-8}
        & & 
        $n=7$ & $n>7$ &
        $n=8$ & $n>8$ &
        $n=9$ & $n>9$\\
    
        \hline\hline
        
        %
        
        $f_2(n)$ & 
        $\binom{n}{2}$ & 
        $19$ & $4n-8$ & 
        \multicolumn{2}{c||}{$4n-8$} &
        $27$ & $4n-8$ \\
        
        \hline
        
        $f_3(n)$ & 
        $\binom{n}{3}$ & 
        \multicolumn{2}{c||}{$\frac{19n-58}{3}$} & 
        $32$ & $\frac{19n-62}{3}$ & 
        \multicolumn{2}{c|}{$\frac{19n-54}{3}$} \\
        
        \hline
        
        $f_4(n)$ & 
        $\binom{n}{4}$ & 
        \multicolumn{2}{c||}{$5n-19$} & 
        \multicolumn{2}{c||}{$5n-20$} & 
        \multicolumn{2}{c|}{$5n-15$} \\
        
        \hline
        
        $f_5(n)$ & 
        $\binom{n}{5}$ & 
        \multicolumn{2}{c||}{$2n-8$} & 
        \multicolumn{2}{c||}{$2n-9$} & 
        \multicolumn{2}{c|}{$2n-6$} \\
        
        \hline
        
        $f_6(n)$ & 
        $\binom{n}{6}$ & 
        \multicolumn{2}{c||}{$\frac{n-4}{3}$} & 
        \multicolumn{2}{c||}{$\frac{n-5}{3}$} &
        \multicolumn{2}{c|}{$\frac{n-3}{3}$} \\
        
        \hline
        
        %
        
        $f(n)$ & 
        $\sum \limits_{t = 0}^{n} \binom{n}{t}$ &
        \multicolumn{2}{c||}{$\frac{56n - 176}{3}$} &
        \multicolumn{2}{c||}{$\frac{56n - 184}{3}$} &
        \multicolumn{2}{c|}{$\frac{56n - 144}{3}$} \\
        
        \hline
    \end{tabular}
    \end{spacing}
    \caption{The maximum number~${f_t(n)}$ of cliques of size~$t$ 
        and the maximum number~${f(n)}$ of cliques in $n$-vertex $1$-planar graphs.}
\end{table}

\subsection{Related extremal results}
The problem of maximising cliques in a given graph class has long been an important topic in extremal combinatorics. 
Motivated by the famous problem of Erd\H{o}s of maximising pentagons in triangle-free graphs (solved by Hatami~\emph{et al.}~\cite{c5c3} and Grzesik~\cite{MR2959390}), the dual problem problem of maximising triangles in the class of pentagon-free graphs was introduced by Bollob\'as and Gy\H{o}ri~\cite{MR2433860}, and has still not been resolved. 
There are now many graphs~$H$ for which the problem of maximising triangles in $H$-free graphs has been studied, see for example~\cites{AlonShikhelman,MR2900057}.
Determining the maximum number of cliques in a regular graph was a famous open problem of Alon and Kahn (phrased in terms of independent sets). 
Kahn~\cite{MR1841642} resolved the problem under the assumption that the graph is bipartite. 
The conjecture was finally settled by Zhao~\cite{MR2593625} who was able to deduce the general case from the bipartite case.
 The maximum number of cliques in the class of graphs with bounded clique or independence number was determined by Cutler and Radcliffe~\cite{MR2831105}.  If only the number of edges and vertices are fixed, then Wood~\cite{wood} determined the maximum possible number of cliques.

A more topological example is the class of graphs embeddable on a fixed surface, where Huynh, Joret and Wood~\cite{HJW-subgraphdensities} determined asymptotically the maximum number of triangles and the maximum number of subgraphs isomorphic to~$K_4$. 
In the case of~$K_t$ for~${t \geq 5}$ they obtained estimates on the maximum number of subgraphs isomorphic to~$K_t$ (which for~${t \geq 5}$ is constant with respect to~$n$).
Dujmovi\'{c}~\emph{et al}.~\cite{MR2838013} considered the question of the maximising the total number of cliques in graphs embeddable on a fixed surface, determining the precise asymptotic dependence on $n$, and giving a structural description of the class of extremal graphs.

Other minor closed classes have also been considered. Wood~\cite{MR3558055} determined for all $t$ and $r \le 9$ the maximum number of subgraphs isomorphic to $K_t$ in the class of $K_r$-minor free graphs, as well as the maximum total number of cliques in these classes. 
While giving bounds on the number of graphs in proper minor closed classes, Norine, Seymour, Thomas and Wollan~\cite{MR2236510} obtained a bound on the total number of cliques in a graph with no $K_r$-minor (of the form ${c \cdot n}$).
Bounds of this form were also obtained independently by Wood~\cite{wood}, and used in an algorithmic problem about finding separators by Reed and Wood~\cite{MR2571902}.
Fomin, Oum and Thilikos~\cite{MR2673004} improved the constant factor~$c$ to $2^{O(\log(\log(r)))}$. 
This constant was further improved to $2^{5r+o(r)}$ by Lee and Oum~\cite{MR3414466}, resolving a conjecture of Wood~\cite{wood},  and finally to $3^{2r/3+o(r)}$ by Fox and Wei~\cite{MR3667668}, which is sharp for~${n \geq 4r/3}$. 
This result was subsequently strengthened by Fox and Wei~\cite{MR4187140} to hold in the class of $K_r$-subdivision free graphs.
Kawarabayashi and Wood~\cite{kawa} considered the class of graphs forbidding a given odd-minor and proved a bound of $O(n^2)$ on the number of cliques, which is tight up to a constant factor.

In this paper we emphasize only the extremal problem of bounding the number of cliques in graph classes. However, the algorithmic problem of enumerating the cliques (or other subgraphs) in graphs from a given graph class has also garnered much attention. 
In particular, Chiba and Nishizeki~\cite{MR774940} gave an algorithm to list the cliques in any graph, whose running time is linear in the number of cliques and arboricity of the graph.

\subsection{Overview of the proofs}
We will prove Theorems~\ref{thm:mainintro1},~\ref{thm:mainintro2} and~\ref{thm:mainintro3} by proving the stronger results Theorems~\ref{thm:1planartriangles},~\ref{thm:1planarbigcliques} and~\ref{thm:1planarallcliques}
stated in Section~\ref{sec:mainproofs}, which additionally give a precise characterisation of the extremal graphs.
We now give a rough overview of the proof of Theorem~\ref{thm:1planartriangles}. 
The proofs of Theorems~\ref{thm:1planarbigcliques} and~\ref{thm:1planarallcliques} work in a similar way.
We proceed by induction on the number of vertices~$n$ of an extremal graph~$G$. 
When~$n$ is at most twenty, we determine all possible extremal graphs with some computational assistance. 

For larger values of~$n$, we fix a $1$-drawing of $G$ satisfying some useful properties (see Section~\ref{sec:rich1drawings}).
Our initial strategy will be to iteratively find and delete crossed edges which are contained in at most~$3$ triangles. 
If we can reduce~$G$ to a planar graph in this way, then we can bound the number of triangles in~$G$ in terms of the crossing number and the extremal number of triangles in planar graphs to complete the proof.
Otherwise, $G$ contains a subgraph such that every crossed edge is in at least four triangles. 
Using the properties of the drawing and the simple observation that vertex-disjoint triangles cross an even number of times, we are able to deduce that $G$ contains a subgraph isomorphic to either~$K_{2,2,2,2}$ or~$K_6$. 
It turns out that no $3$-connected $1$-planar graph contains a proper subgraph isomorphic to~$K_{2,2,2,2}$, and it is easy to show that~$G$ is $3$-connected, so we deduce that $G$ has a subgraph isomorphic to~$K_6$. 
By analysing the drawing restricted to this subgraph, we find a triangle which separates~$G$. 
This allows us to inductively deduce the structure of~$G$ from the structure of the two sides of the separation, and thus compute the extremal function.

\subsection{Structure of the paper}
The paper is structured as follows. 
In Section~\ref{sec:prelims}, we formally define drawings, set up the notation which will be used throughout the paper and collect some useful results from the literature. 
In Section~\ref{sec:stitchings}, we present the classes of extremal graphs for Theorems~\ref{thm:mainintro1},~\ref{thm:mainintro2} and~\ref{thm:mainintro3} and briefly discuss some properties of these constructions. 
In Section~\ref{sec:rich1drawings}, we describe a class of $1$-drawings with some useful properties, and prove interesting facts about them. 
In Section~\ref{sec:nok6}, we compute an upper bound for the number of cliques of any given size in a $1$-planar graph with no subgraph isomorphic to~$K_6$. In Section~\ref{sec:mainproofs}, we complete the proofs of Theorems~\ref{thm:mainintro1}, \ref{thm:mainintro2} and~\ref{thm:mainintro3}, and additionally we show that the classes of graphs defined in Section~\ref{sec:stitchings} are exactly the classes of extremal graphs.
In Section~\ref{sec:treedec}, we give more precise characterisations of the classes of extremal graphs in terms of tree-decompositions, allowing us to generate the graphs in each class in polynomial time. We conclude in Section~\ref{sec:conclusion}
with a discussion of some open questions.

\section{Preliminaries}\label{sec:prelims}

\subsection{Basic notation}

Given an integer~$k$, we denote by~${[k]}$ the set of all integers~$i$ with~${1 \leq i \leq k}$. 

Let~$G$ be a graph with vertex set~${V(G)}$ and edge set~${E(G)}$. 
We denote an edge between vertices~$v$ and~$w$ by the string~$vw$. 
For a set~${X \subseteq V(G)}$, we denote the subgraph of $G$ \emph{induced} on~$X$ by~${G[X]}$. 
We denote by~${G - X}$ the subgraph of~$G$ induced on~${V(G) \setminus X}$.
We denote by~$\overline{G}$ the \emph{complement} of~$G$, which is the graph with vertex set~${V(G)}$ and edge set~${\{ vw \colon v, w \in V(G), vw \notin E(G) \}}$.

Let~$G$ and~$H$ be graphs. 
We denote by ${G \cap H}$ the \emph{intersection} of~$G$ and~$H$, which is the graph on~${V(G) \cap V(H)}$ with the edge set~${E(G) \cap E(H)}$.  
We denote by~${G \cup H}$ the \emph{union} of~$G$ and~$H$, which is the graph on~${V(G) \cup V(H)}$ with the edge set~${E(G) \cup E(H)}$, 
and we denote by~${G \sqcup H}$ the \emph{disjoint union} of~$G$ and~$H$, which is the graph on the disjoint union of~$V(G)$ and~$V(H)$ whose edge set is the disjoint union of~$E(G)$ and~$E(H)$. 
We denote by~${G + H}$ the \emph{graph join} of~$G$ and~$H$, which is the union of~${G \sqcup H}$ with the complete bipartite graph with bipartition classes~$V(G)$ and~$V(H)$. 
We denote by~${G \boxtimes H}$ the \emph{strong product} of~$G$ and~$H$, which is the graph on~${V(G) \times V(H)}$ with an edge between~$(v,w)$ and~$(v',w')$ if and only if either
\begin{itemize}
    \item ${v = v'}$ and~${ww' \in E(H)}$; 
    \item ${vv' \in E(G)}$ and~${w = w'}$; or
    \item ${vv' \in E(G)}$ and~${ww' \in E(H)}$.
\end{itemize}
We write~${H \subseteq G}$ if~$H$ is a subgraph of~$G$, that is~$V(H)$ is a subset of~$V(G)$ and~$E(H)$ is a subset of~$E(G)$. 
We write~$\mathcal{N}(G,H)$ for the number of subgraphs~${H' \subseteq G}$ that are isomorphic to~$H$.

For a non-negative integer~$k$, we denote by~$K_k$ a fixed clique (complete graph) with~$k$ vertices, and if~${k \geq 1}$ we denote by~$P_k$ a fixed path with $k$ vertices, and if~${k \geq 3}$ we denote by~$C_k$ a fixed cycle with~$k$ vertices.
Given an integer~${\ell \geq 2}$ and a family~${( k_i \colon i \in [\ell])}$ of positive integers, we denote by~$K_{k_1, \ldots, k_\ell}$ a fixed complete multipartite graph with~$\ell$ partition classes of sizes~${k_1,\ldots, k_\ell}$, respectively. 

A \emph{separation} of~$G$ is a pair~${(A, B)}$ of subsets of~$V(G)$ such that~${G[A] \cup G[B] = G}$. 
Its \emph{order} is the size of its \emph{separator}~${A \cap B}$, and we call a separation of order~$k$ for some integer~$k$ a \emph{$k$-separation}. 
We say a separation~$(A,B)$ is \emph{non-trivial} if both~${A \setminus B}$ and~${B \setminus A}$ are non-empty. 
For a positive integer~$k$, we say~$G$ is \emph{$k$-connected} if~$G$ has at least~${k+1}$ vertices and no non-trivial $k$-separation. 
A subset~$S$ of~${V(G)}$ is a \emph{non-trivial $\abs{S}$-separator of~$G$} if~${G-S}$ has at least two components, i.e.~there is a non-trivial separation of~$G$ with~$S$ as its separator. 
Given subsets~$S$,~$X$ and~$Y$ of~$V(G)$, we say~$S$ \emph{separates}~$X$ and~$Y$ if there is a separation~${(A,B)}$ whose separator is~$S$ with~${X  \subseteq A}$ and~${Y \subseteq B}$. 
We call a triangle~$C$ in~$G$ for which~${V(C)}$ is a non-trivial $3$-separator a \emph{separating triangle of~$G$}.

\subsection{Drawings}\label{subsec:drawings}

Given a graph~$G$, 
let us consider the simplicial $1$-complex~$\abs{G}$ of~$G$ (as a topological space). 
For a surface space~$\mathbb{S}$ 
a \emph{drawing of~$G$ in~$\mathbb{S}$} is a map~${\phi \colon \abs{G} \to \mathbb{S}}$ with the following properties: 
\begin{enumerate}
    [label=(D\arabic*)]
    \item\label{item:drawing1} $\phi$ is continuous; 
    \item\label{item:drawing2} $\phi$ is injective on the vertices of~$G$, i.e.~for every vertex~$v$, if~${\phi(v) = \phi(x)}$, then~${x = v}$;
    \item\label{item:drawing3} there are only finitely many points~${x \in \mathbb{S}}$ for which $\phi^{-1}(x)$ has size least~$2$; we call these points the \emph{crossings (of~$\phi$)}; 
    \item\label{item:drawing4} each crossing $x$ of $\phi $ satisfies $\abs{\phi^{-1}(x)} = 2$, 
    and the corresponding points of~$\abs{G}$ are inner points of distinct edges; we call these edges the \emph{edges involved in the crossing}; 
    \item\label{item:drawing5} for each crossing~$x$ of~$\phi$ there is an open set~$D$ with~${x \in D}$ such that for every open set~${D' \subseteq D}$ with~${x \in D'}$ the removal of the image of one of the edges involved in~$x$ disconnects~$D'$ such that multiple components contain image points of the other edge involved in~$x$. 
\end{enumerate}
For two edges involved in a crossing of~$\phi$, we say these edges \emph{cross} (with respect to~$\phi$). 
An edge which crosses some other edge is called a \emph{crossed edge} (with respect to~$\phi$). 
We may omit the phrase ``with respect to~$\phi$'' if the drawing we are referring to is clear from context. 

Intuitively we think of a drawing as the image of such a map~$\phi$ in the surface: 
a representation of a graph in the surface, where the vertices of the graph are represented by distinct points and the edges by Jordan arcs joining the corresponding pairs of points. 
Properties~\ref{item:drawing3} and~\ref{item:drawing4} mean that there are finitely many points where exactly two edges cross. 
Property~\ref{item:drawing5} means that for each crossing point the edges really ``cross'' and not just ``touch''. 

\vspace{0.2cm}

Let~$\phi$ be a drawing of a graph~$G$ (in a surface~$\mathbb{S}$). 
For simplicity, we may refer to the images of vertices or edges simply as vertices or edges, respectively, of the drawing (and similarly for other graph structures). 
For a subgraph~${H}$ of~$G$ we write~${\phi \restricted H}$ for the drawing of~$H$ in~$\mathbb{S}$ which is just the restriction of~$\phi$ to~$\abs{H}$ (considered as a subspace of~$\abs{G}$). 
We write~${\phi(H)}$ instead of~${\phi(\abs{H})}$ for the image of~$\abs{H}$, and for simplicity, we will write~${\phi(e)}$ instead of~${\phi(G[v,w])}$ for an edge~${e = vw}$. 
Given another drawing~$\phi'$ of~$G$ in~$\mathbb{S}$, we say~$\phi$ and~$\phi'$ are \emph{equivalent} if~$\phi'$ can be written as the composition of~$\phi$ and an automorphism of~$\mathbb{S}$. 
Moreover, we say~$\phi$ and~$\phi'$ are \emph{weakly equivalent} if~$\phi$ is equivalent to the composition of~$\phi'$ with some automorphism of~$G$. 

A \emph{region}~$R$ of~${\mathbb{S} \setminus \phi(G)}$ is an (arc-)connected component of that topological space. 
The \emph{boundary} of a region~$R$ is the difference of the closure of~$R$ (in~$\mathbb{S}$) and~$R$. 
If the boundary of~$R$ does not contain any crossings of~$\phi$, then we call~$R$ a \emph{face} of~$\phi$. 
Note that if the boundary of a face contains an inner point of some edge~$e$ of~$G$, then it contains~$\phi(e)$. 
We define the \emph{degree} of a face~$F$ to be the number of edges~$e$ for which~$\phi(e)$ is contained in the boundary of~$F$, plus the number of edges~$e$ for which~$\phi(e)$ is contained in the boundary of~$F$ and not contained in the boundary of any other region of~${\mathbb{S} \setminus \phi(G)}$. 
If the boundary of a face is equal to~$\phi(C)$ for a cycle~${C \subseteq G}$, then we call~$C$ a \emph{facial cycle} with respect to~$\phi$. 
Note that if~$G$ is $2$-connected, then the boundary of each face of~$\phi$ is a facial cycle with respect to~$\phi$. 
If the drawing is clear from the context, we may refer to the faces of the drawing or facial cycles with respect to the drawing by faces or facial cycles, respectively, of the graph. 

We say~$\phi$ is \emph{simple} if no two adjacent edges cross. 
Given a non-negative integer~$k$, we say~$\phi$ is a \emph{$k$-drawing (of~$G$ in~$\mathbb{S}$)} if no edge of~$G$ is involved in more than~$k$ crossings. 

\begin{remark}
    A graph has a $1$-drawing (in~$\mathbb{S}$) if and only if it has a simple $1$-drawing (in~$\mathbb{S}$). \qed
\end{remark}

If~$\phi$ is a drawing of a graph~$G$ in the $2$-dimensional sphere~$\mathbb{S}^2$, we call~$\phi$ just a \emph{drawing of~$G$}.
From now on we will only consider such drawings. 

Note that there is a natural correspondence between drawings in the plane and in~$\mathbb{S}^2$.

\vspace{0.2cm}

Given a non-negative integer~$k$, a graph~$G$ is called \emph{$k$-planar} if there is a $k$-drawing~$\phi$ of~$G$. 
Note that the~$0$-planar graphs in this context are precisely the planar graphs. 
The minimum integer~$k$ for which a graph~$G$ is $k$-planar is the \emph{local crossing number} of~$G$. 

A graph~$G$ is an \emph{edge-maximal $k$-planar graph} if no proper supergraph of~$G$ on the same vertex set is $k$-planar. 

Building on earlier work of Schumacher~\cite{Schumacher86}, Suzuki~\cite{suzuki2010} studied which edge-maximal $1$-planar graphs have unique $1$-drawings, up to weak equivalence. 
We will use the following lemma.
 
\begin{lemma}[\cite{suzuki2010}*{Lemma~17 and Corollary~4}]
    \label{lem:unique1drawing}
    The graphs~$K_6$,~$K_3+C_4$ and~$K_{2,2,2,2}$ each have a unique simple $1$-drawing up to weak equivalence.
\end{lemma}

    \begin{figure}[htbp]
        \centering
        \begin{tikzpicture}
            [scale=0.8]
            \tikzset{vertex/.style = {circle, draw, fill=black!50, inner sep=0pt, minimum width=4pt}}
            \tikzset{edge0/.style = {line width=1.2pt, black}}
            \tikzset{edge1/.style = {line width=1pt, blue, opacity=0.7}}
            
            \node [vertex] (v0) at (210:3) {};
            \node [vertex] (v1) at (330:3) {};
            \node [vertex] (v2) at (90:3) {};
            \node [vertex] (w0) at (210:1.5) {};
            \node [vertex] (w1) at (330:1.5) {};
            \node [vertex] (w2) at (90:1.5) {};
            
            \foreach \i in {0,1,2} {
                \pgfmathtruncatemacro{\j}{Mod(\i+1,3)}
                \foreach \v in {v,w} {
                    \draw [edge0] (\v\i) edge (\v\j) {};
                };
                \draw [edge0] (v\i) edge (w\i) {};
                \draw [edge1] (v\i) edge (w\j) {};
                \draw [edge1] (v\j) edge (w\i) {};
            };
        \end{tikzpicture}
        \quad
        \begin{tikzpicture}
            [scale=0.8]
            \tikzset{vertex/.style = {circle, draw, fill=black!50, inner sep=0pt, minimum width=4pt}}
            \tikzset{edge0/.style = {line width=1.2pt, black}}
            \tikzset{edge1/.style = {line width=1pt, blue, opacity=0.7}}
            
            \coordinate (v0) at (90:3) {};
            \coordinate (v1) at (210:3) {};
            \coordinate (v2) at (330:3) {};
            \coordinate (v3) at (160:0.75) {};
            \coordinate (v4) at (20:0.75) {};
            \coordinate (v5) at (210:1.5) {};
            \coordinate (v6) at (330:1.5) {};
            
            \foreach \i in {1,...,6} {
                \pgfmathtruncatemacro{\p}{Mod(\i,2)}
                \draw [edge\the\numexpr 1-\p \relax] (v\i) edge (v\the\numexpr \i+\p \relax) {};
                   \ifthenelse{\i < 5}{
                       \draw [edge0] (v0) edge (v\i) {};
                       \draw [edge0] (v\i) edge (v\the\numexpr 6-\p \relax) {};
                       \draw [edge1] (v\i) edge (v\the\numexpr 5+\p \relax) {};
                   }{
                       \draw [edge1] (v0) edge (v\i) {};
                       \draw [edge1] (v\the\numexpr \i-4 \relax) edge (v\the\numexpr \i-2 \relax) {};
                   }
            }
            
            \foreach \i in {0,...,6} {
                \node [vertex] at (v\i) {};
            };
        \end{tikzpicture}
        \quad
        \begin{tikzpicture}
            [scale=0.8]
            \tikzset{vertex/.style = {circle, draw, fill=black!50, inner sep=0pt, minimum width=4pt}}
            \tikzset{edge0/.style = {line width=1.2pt, black}}
            \tikzset{edge1/.style = {line width=1pt, blue, opacity=0.7}}
            
            \node [vertex] (v0) at (45:2.5) {};
            \node [vertex] (v1) at (135:2.5) {};
            \node [vertex] (v2) at (225:2.5) {};
            \node [vertex] (v3) at (315:2.5) {};
            \node [vertex] (w0) at (45:1) {};
            \node [vertex] (w1) at (135:1) {};
            \node [vertex] (w2) at (225:1) {};
            \node [vertex] (w3) at (315:1) {};
            
            \foreach \i in {0,1,2,3} {
                \pgfmathtruncatemacro{\j}{Mod(\i+1,4)}
                \pgfmathtruncatemacro{\k}{Mod(\i+2,4)}
                \foreach \v in {v,w} {
                    \draw [edge0] (\v\i) edge (\v\j) {};
                };
                \draw [edge0] (v\i) edge (w\i) {};
                \draw [edge1] (v\i) edge (w\j) {};
                \draw [edge1] (v\j) edge (w\i) {};
                \draw [edge1] (w\i) edge (w\k) {};
            };
            \draw [edge1] plot[smooth, tension=0.7] coordinates {(v0) (135:3) (v2)};
            \draw [edge1] plot[smooth, tension=0.7] coordinates {(v1) (45:3) (v3)};
        \end{tikzpicture}
        \caption{The unique simple $1$-planar drawings (up to weak equivalence) of~$K_6$ on the left, of~$K_3 + C_4$ in the middle and of~$K_{2,2,2,2}$ on the right.}
        \label{fig:unique1drawing}
    \end{figure}
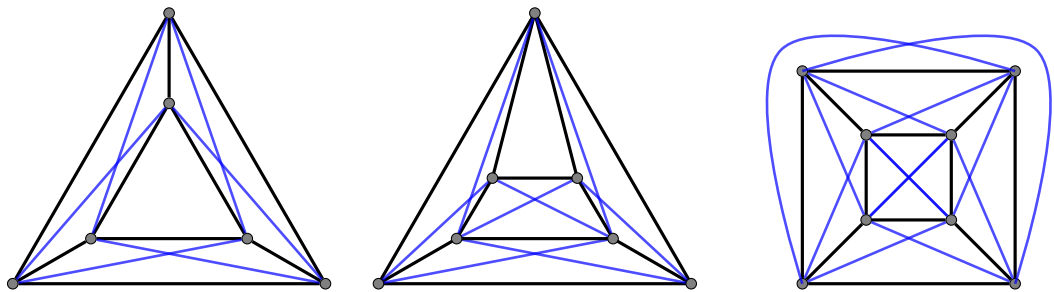

From Lemma~\ref{lem:unique1drawing} and Figures~\ref{fig:unique1drawing} and~\ref{fig:7vx1pl} we obtain the following corollary. 

\begin{corollary}
    \label{cor:drawing-properties}
    Let~$G$ be isomorphic to a graph in $\{K_6,K_3+C_4,K_{2,2,2,2}\}$ and let $\phi$ be a simple $1$-drawing of $G$. The following statements hold.
    \begin{enumerate}
        [label=(\roman*)]
        \item \label{item:dp1} Every edge which is not crossed with respect to~$\phi$ is incident with two edges which cross each other.
        \item \label{item:dp2} If the boundary of some region~$R$ of ${\mathbb{S}^2\setminus \phi(G)}$ contains at least three vertices, then the boundary of~$R$ is a facial triangle with respect to~$\phi$, and~$G$ is not isomorphic to~$K_{2,2,2,2}$.
        \item \label{item:dp3} If $G$ is not isomorphic to $K_{2,2,2,2}$, then exactly two triangles are facial with respect to $\phi$ and these are vertex-disjoint. 
        \qed
    \end{enumerate}
\end{corollary}

Given a drawing~$\phi$ of a graph~$G$, we define the \emph{true-planar skeleton~$\mathcal{S}(\phi)$ of~$\phi$} to be the 
subgraph of~$G$ containing exactly the edges of~$G$ which are not crossed. 
We will always consider~$\mathcal{S}(\phi)$ to be drawn with the restriction~${\phi \restricted \mathcal{S}(\phi)}$, which is a planar drawing of~${\mathcal{S}(\phi)}$. 
For all figures in this paper we use black edges for the true-planar skeletons of the drawings depicted.

\subsection{%
\texorpdfstring{Extremal results for planar and $1$-planar graphs}%
{Extremal results for planar and 1-planar graphs}}

We now collect some results from the literature which will be needed in the upcoming proofs.

A planar graph $G$ is an \emph{Apollonian network} if it is isomorphic to $K_3$ or it contains a vertex $v$ of degree $3$ such that $G-v$ is an Apollonian network. 

\begin{theorem}
    \label{thm:planarK3}
    \cite{MR519175}
    For an integer~${n \geq 3}$, every $n$-vertex planar graph has at most~${3n-8}$ triangles, with equality if and only if it is an Apollonian network.
\end{theorem}

\begin{theorem}
    \label{thm:planarK4}
    \cites{MR791009,wood} 
    For an integer~${n \geq 4}$, every $n$-vertex planar graph has at most~${n-3}$ subgraphs isomorphic to~$K_4$, with equality if and only if it is an Apollonian network. 
\end{theorem}

\begin{thm}
    \label{thm:edgeextremal1planar}
    \cite{BodendiekSW83}
    Given a non-negative integer~$n$, the maximum number~$f_2(n)$ of edges in a $1$-planar graph with~$n$ vertices is given by 
    
    \[
        f_2(n) = 
        \begin{cases}
            \binom{n}{2} & \textnormal{ if } n \leq 6,\\
            4n-9         & \textnormal{ if } n \in \{7,9\},\\
            4n-8         & \textnormal{ otherwise. }
        \end{cases}
    \]
\end{thm}

\begin{lemma}
    \label{lem:crossingnumber}
    \cite{ZLW12}
    A simple $1$-drawing of an $n$-vertex graph has at most~${n-2}$ crossings.
\end{lemma}

\begin{thm}
    \label{thm:Korzhik}
    \cite{Korzhik08}
    A $7$-vertex graph is $1$-planar if and only if it has no subgraph isomorphic to~${K_4 + \overline{K_3}}$.
\end{thm}

The following corollary is simple to deduce.
\begin{corollary}
    \label{cor:7vertex1planar}
    A graph with at most~$7$ vertices is $1$-planar if and only if it is isomorphic to a subgraph of either~${K_3 + C_4}$ or~${K_3 + \overline{K_{1,3}}}$. 
    \qed
\end{corollary}

    \begin{figure}[htbp]
        \centering
        \begin{tikzpicture}
            [scale=0.8]
            \tikzset{vertex/.style = {circle, draw, fill=black!50, inner sep=0pt, minimum width=4pt}}
            \tikzset{edge0/.style = {line width=1.2pt, black}}
            \tikzset{edge1/.style = {line width=1pt, blue, opacity=0.7}}
            
            \coordinate (v0) at (90:3) {};
            \coordinate (v1) at (210:3) {};
            \coordinate (v2) at (330:3) {};
            \coordinate (v3) at (160:0.75) {};
            \coordinate (v4) at (20:0.75) {};
            \coordinate (v5) at (210:1.5) {};
            \coordinate (v6) at (330:1.5) {};
            
            \foreach \i in {1,...,6} {
                \pgfmathtruncatemacro{\p}{Mod(\i,2)}
                \draw [edge\the\numexpr 1-\p \relax] (v\i) edge (v\the\numexpr \i+\p \relax) {};
                   \ifthenelse{\i < 5}{
                       \draw [edge0] (v0) edge (v\i) {};
                       \draw [edge0] (v\i) edge (v\the\numexpr 6-\p \relax) {};
                       \draw [edge1] (v\i) edge (v\the\numexpr 5+\p \relax) {};
                   }{
                       \draw [edge1] (v0) edge (v\i) {};
                       \draw [edge1] (v\the\numexpr \i-4 \relax) edge (v\the\numexpr \i-2 \relax) {};
                   }
            }
            
            \foreach \i in {0,...,6} {
                \node [vertex] at (v\i) {};
            };
        \end{tikzpicture}
        \quad
        \begin{tikzpicture}
            [scale=0.8]
            \tikzset{vertex/.style = {circle, draw, fill=black!50, inner sep=0pt, minimum width=4pt}}
            \tikzset{edge0/.style = {line width=1.2pt, black}}
            \tikzset{edge1/.style = {line width=1pt, blue, opacity=0.7}}
            
            \node [vertex] (v0) at (210:3) {};
            \node [vertex] (v1) at (330:3) {};
            \node [vertex] (v2) at (90:3) {};
            \node [vertex] (w0) at (210:1.5) {};
            \node [vertex] (w1) at (330:1.5) {};
            \node [vertex] (w2) at (90:1.5) {};
            \node [vertex] (c) at (0:0) {};
            
            \foreach \i in {0,1,2} {
                \pgfmathtruncatemacro{\j}{Mod(\i+1,3)}
                \foreach \v in {v,w} {
                    \draw [edge0] (\v\i) edge (\v\j) {};
                };
                \draw [edge0] (v\i) edge (w\i) {};
                \draw [edge1] (v\i) edge (w\j) {};
                \draw [edge1] (v\j) edge (w\i) {};
                \draw [edge0] (w\i) edge (c) {};
            };
        \end{tikzpicture}
        \caption{A simple $1$-planar drawing of~${K_3 + C_4}$ on the left and of~${K_3 + \overline{K_{1,3}}}$ on the right. }
        \label{fig:7vx1pl}
    \end{figure}
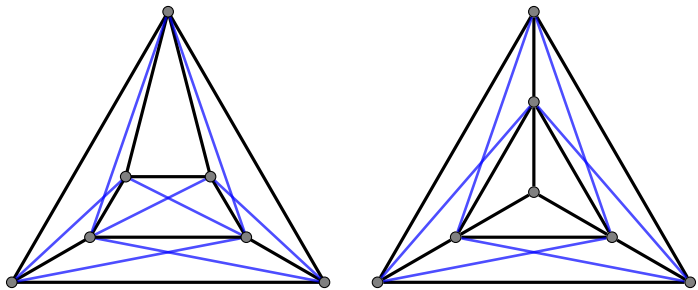

\section{Constructing the extremal graphs}\label{sec:stitchings}

In this section, we will construct the extremal graphs for Theorems~\ref{thm:mainintro1},~\ref{thm:mainintro2} and~\ref{thm:mainintro3}. 
In order to do this, we must first give some definitions. 
The following simple lemma, which we state here without proof, describes a method of combining two $1$-planar graphs to form a larger $1$-planar graph which is fundamental to our constructions.

\begin{lemma}
    \label{lem:stitching}
    Let~$G_1$ and~$G_2$ be $1$-planar graphs such that~${G_1 \cap G_2}$ is a clique~$K$ with at most three vertices.  
    For each~${i \in [2]}$ let~$\phi_i$ be a $1$-drawing of~$G_i$  
    such that~$K$ is contained in~$\mathcal{S}(\phi_i)$
    and~${\phi_i(G_i - V(K))}$ is contained in a unique region of~${\mathbb{S}^2 \setminus \phi_i(K)}$. 
    Then~${G := G_1 \cup G_2}$ has a $1$-drawing~$\phi$ whose restriction to~$G_i$ is equivalent to~$\phi_i$ for~${i \in [2]}$, such that no edge of~$G_1$ crosses any edge of~$G_2$. 
    \qed
\end{lemma}

This lemma motivates the following definition. 

\begin{definition}
    Given~$1$-planar graphs~$H_1$ and $H_2$, a graph~$G$ is a \emph{stitch of~$H_1$ and~$H_2$} if there is a simple $1$-drawing~$\phi$ of~$G$ and there are proper subgraphs~$G_1$ and~$G_2$ of~$G$ isomorphic to~$H_1$ and~$H_2$ respectively, such that~${G_1 \cup G_2 = G}$ and~${T:=G_1 \cap G_2}$ is a facial triangle with respect to both~${\phi \restricted G_1}$ and~${\phi \restricted G_2}$. We say that~$(G_1,G_2,T,\phi)$ \emph{witnesses} that $G$ is a stitch of $G_1$ and $G_2$.
\end{definition}

Lemma~\ref{lem:stitching} guarantees that the stitch of two $1$-planar graphs is $1$-planar. 
Furthermore, given graphs~$H_1$ and~$H_2$ with $1$-drawings~$\phi_1$ and~$\phi_2$ respectively, every stitch of~$H_1$ and~$H_2$ has a $1$-drawing~$\phi$ such that all but one facial triangle of~$\phi_1$ and all but one facial triangle of~$\phi_2$ corresponds to a facial triangle of~$\phi$. 
This allows us to iteratively take stitches of copies of a graph~$H$, provided~$H$ has a $1$-drawing with at least two facial triangles. For example, we observe that the strong product of a triangle and a path is $1$-planar by iteratively stitching copies of~$K_6$ together. 

On the other hand, we can identify whether a graph is a stitch of two smaller graphs by finding a $1$-drawing for which the true-planar skeleton contains a  separating triangle, as the following lemma illustrates. 

\begin{lemma}\label{lem:separatingstitch}
    Given a graph~$G$, there exist graphs~$H_1$ and~$H_2$ such that~$G$ is a stitch of~$H_1$ and~$H_2$ if and only if there is a simple $1$-drawing~$\phi$ of~$G$ such that~${\mathcal{S}(\phi)}$ contains a separating triangle of~$G$. 
    \qed
\end{lemma}

\begin{figure}[htbp]
	\centering
	\begin{tikzpicture}
        [scale=0.8]
    	\tikzset{vertex/.style = {circle, draw, fill=black!50, inner sep=0pt, minimum width=4pt}}
            \tikzset{edge0/.style = {line width=1.2pt, black}}
            \tikzset{edge1/.style = {line width=1pt, blue, opacity=0.7}}

        \coordinate (v0) at (90:3) {};
        \coordinate (v1) at (210:3) {};
        \coordinate (v2) at (330:3) {};
        \coordinate (v3) at (160:0.75) {};
        \coordinate (v4) at (20:0.75) {};
        \coordinate (v5) at (210:1.5) {};
        \coordinate (v6) at (330:1.5) {};
        \coordinate (v7) at (160:0.25){};
        \draw[edge0] (v0)--(v1) {};
        \draw[edge0] (v1)--(v2) {};
        \draw[edge0] (v2) --(v0) {};
        \draw[edge0] (v0)--(v3) {};
        \draw[edge0] (v0) -- (v4) {};
        \draw[edge0] (v5)--(v6) {};
        \draw[edge0] (v3)--(v5) {};
        \draw[edge0] (v4) --(v6) {};
        \draw[edge0] (v1)--(v5) {};
        \draw[edge0] (v2) -- (v6) {};
        \draw[edge0] (v3)--(v7) {};
        \draw[edge0] (v4) -- (v7) {};
        \draw[edge0] (v7)--(v5) {};
        
        \draw[edge1] (v0)--(v5) {};
        \draw[edge1] (v0) -- (v6) {};
        \draw[edge1] (v4)--(v2) {};
        \draw[edge1] (v4) -- (v5) {};
        \draw[edge1] (v7)--(v6) {};
        
        \draw[edge1] (v3)--(v1) {};
        \draw[edge1] (v0) -- (v7) {};
        \draw[edge1] (v4)--(v3) {};
        \draw[edge1] (v2) -- (v5) {};
        \draw[edge1] (v1)--(v6) {};
	
        \foreach \i in {0,...,7} {
            \node [vertex] at (v\i) {};
        };
	\end{tikzpicture}
	\caption{A simple $1$-drawing of~${K_2 + \overline{P_6}}$. We will show in Subsection~\ref{subsec:efficent} that this drawing is unique (up to weak equivalence).}
	\label{fig:K2+P5C}
\end{figure}
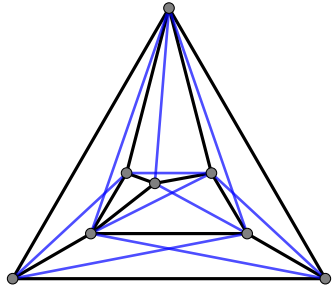

We now define classes~$\mathcal{E}_t$ for~${t \in \{3,4,5,6\}}$ and~$\mathcal{E}$, which we will show are the classes of extremal~$1$-planar graphs for cliques of size~$t$ and for the total number of cliques, respectively.

\begin{definition}\label{def:extremalgraphs}
    We define the classes~$\mathcal{E}_3$,~$\mathcal{E}_4$,~$\mathcal{E}_5$,~$\mathcal{E}_6$ and~$\mathcal{E}$ recursively as follows. 
    Let~$\mathcal{E}_3$ be the class consisting of 
    \begin{itemize}
        \item every graph with at most two vertices; 
        \item every graph isomorphic to one of~$K_3$,~$K_4$,~$K_5$,~$K_6$,~$K_3 + C_4$ or~$K_{2,2,2,2}$;
        \item every graph~$G$ which is a stitch of some graphs~${H_1, H_2 \in \mathcal{E}_3 \cup \{ K_2 + \overline{P_6} \}}$ such that for each~${i \in [2]}$ there exist integers~${k_i \geq 2}$ and~${s_i \in \{0,1,2\}}$ with~${s_1 + s_2 \leq 2}$ and~${\abs{V(H_i)} = 3k_i + s_i}$.
    \end{itemize}
     
    Let~$\mathcal{E}_4$ be the class consisting of 
    \begin{itemize}
        \item every graph with at most three vertices; 
        \item every clique with at most six vertices;
        \item every graph isomorphic to~$K_3 + C_4$;
        \item every graph~$G$ which is a stitch of some graphs~${H_1, H_2 \in \mathcal{E}_4}$ such that for each~${i \in [2]}$ there exist integers~${k_i \geq 1}$ and~${s_i \in \{0,1,2\}}$ with~${s_1 = 0}$ and \linebreak ${\abs{V(H_i)} = 3k_i + s_i}$.
    \end{itemize}
    
    Let~$\mathcal{E}_5$ be the class consisting of 
    \begin{itemize}
        \item every graph with at most four vertices; 
        \item every graph isomorphic to one of~$K_5$~or~$K_6$;
        \item every graph~$G$ which is a stitch of some graphs~${H_1, H_2 \in \mathcal{E}_5}$ such that for each~${i \in [2]}$ there exist integers~${k_i \geq 1}$ and~${s_i \in \{0,1,2\}}$ with~${s_1 = 0}$ and \linebreak ${\abs{V(H_i)} = 3k_i + s_i}$.
    \end{itemize}
     
    Let~$\mathcal{E}_6$ be the class consisting of 
    \begin{itemize}
        \item every graph with at most five vertices; 
        \item every graph isomorphic to~$K_6$;
        \item every graph~$G$ which is a stitch of some graphs~${H_1, H_2 \in \mathcal{E}_6}$ such that for each~${i \in [2]}$ there exist integers~${k_i \geq 1}$ and~${s_i \in \{0,1,2\}}$ with~${s_1 + s_2 \leq 2}$ and \linebreak ${\abs{V(H_i)} = 3k_i + s_i}$.
    \end{itemize}
    
    Let~$\mathcal{E}$ be the class consisting of 
    \begin{itemize}
        \item every clique with at most six vertices;
        \item every graph isomorphic to~${K_3 + C_4}$; 
        \item every graph~$G$ which is a stitch of some graphs~${H_1, H_2 \in \mathcal{E}}$ such that for each~${i \in [2]}$ there exist integers~${k_i \geq 1}$ and~${s_i \in \{0,1,2\}}$ with~${s_1 = 0}$ and \linebreak ${\abs{V(H_i)} = 3k_i + s_i}$.
    \end{itemize}
\end{definition}

It is worth noting that since stitches are defined in terms of $1$-drawings, the above definitions do not immediately provide an easy way of generating all of the graphs of a given size in any of these classes. In order to generate such a list, one would need a method of determining for each graph in the class which of its triangles are facial with respect to some $1$-drawing, and thus which stitches are possible. In Section~\ref{sec:treedec}, we discuss in more detail how to generate the $n$-vertex graphs in these classes, and provide alternative structural descriptions for them.

\section{%
\texorpdfstring{Rich $1$-drawings and their true-planar skeletons}%
{Rich 1-drawings and their true-planar skeletons}}
\label{sec:rich1drawings}

Consider the following basic observation about $1$-planar graphs. 

\begin{lemma}
    \label{lem:crossingsarecliques}
    \cite{BT18}*{Lemma~3}
	If~$G$ is an edge-maximal $1$-planar graph 
	and some edge~$vw$ crosses some edge~$xy$ in some simple $1$-planar drawing of~$G$, then~${\{v,w,x,y\}}$ induces a~$K_4$ in~$G$.
\end{lemma}

This fact motivates the following definition. 
A $1$-drawing~$\phi$ of a graph~$G$ (in a surface~$\mathbb{S}$) is called \emph{quasi-rich} 
if it is simple and for every pair~$vw$ and~$xy$ of edges which cross, no other edges in ${G[\{v,w,x,y\}]}$ are crossed, and is called \emph{rich} if it is quasi-rich and for every pair~$vw$ and~$xy$ of edges which cross, ${G[\{v,w,x,y\}]}$ is isomorphic to~$K_4$. 

\begin{lemma}
    \label{lem:rich1drawing}
    Let~$\phi$ be a simple $1$-drawing of an edge-maximal $1$-planar graph~$G$ such that~$G$ is $4$-connected or~$\phi$ has the minimum possible number of crossings among all $1$-drawings of~$G$. 
    Then~$\phi$ is rich. 
    In particular, every edge-maximal $1$-planar graph has a rich $1$-drawing.
\end{lemma}

\begin{proof}
    Let~$vw$ and~$xy$ be a pair of edges which cross. 
    By Lemma~\ref{lem:crossingsarecliques}, ${\{v,w,x,y\}}$ induces a~$K_4$ in~$G$. 
    Suppose for a contradiction that~$\phi$ is not rich. 
    Without loss of generality, the edge~$vx$ is crossed by some edge~$e$. 
    Since~$vw$ and~$xy$ cross each other and no other edges, there is a region of~${\mathbb{S}^2 \setminus \phi(G)}$ whose boundary contains both~$v$ and~$x$. 
    Hence the number of crossings could be reduced by redrawing~$vx$ through this region, so we may assume that~$G$ is $4$-connected. 
    However for some endvertex~$z$ of~$e$, the set~${\{v,x,z\}}$ is a $3$-separator of~$G$, a contradiction. 
\end{proof}

Note that not all simple $1$-drawings of edge-maximal $1$-planar graphs are rich, see for example Figure~\ref{fig:simplenotrich}. Lemma~\ref{lem:rich1drawing} and Corollary~\ref{cor:crossingcount}, which we prove at the end of this section, together imply that for $3$-connected edge-maximal $1$-planar graphs, rich $1$-drawings are exactly the $1$-drawings with the fewest crossings. 
However this is not the case for $1$-planar graphs in general. 
Indeed some $1$-planar graphs, such as $K_{3,3}$, contain no cliques of size $4$, and hence have no rich $1$-drawings. 
However it is easily seen that deleting crossed edges from a $1$-drawing preserves richness, which leads to the following observation. 

    \begin{figure}[htbp]
        \centering
        \begin{tikzpicture}
            [scale=0.8]
            \tikzset{vertex/.style = {circle, draw, fill=black!50, inner sep=0pt, minimum width=4pt}}
            \tikzset{edge0/.style = {line width=1.2pt, black}}
            \tikzset{edge1/.style = {line width=1pt, blue, opacity=0.7}}
            
            \coordinate (v0) at (210:3) {};
            \coordinate (v1) at (330:3) {};
            \coordinate (v2) at (90:3) {};
            \coordinate (w0) at (210:1.5) {};
            \coordinate (w1) at (330:1.5) {};
            \coordinate (w2) at (90:1.5) {};
            \coordinate (c) at (0:0) {};
            
            \foreach \i in {0,1,2} {
                \pgfmathtruncatemacro{\j}{Mod(\i+1,3)}
                \draw [edge0] (v\i) edge (v\j) {};
                \ifthenelse{\i = 0}{
                }{
                    \draw [edge0] (w\i) edge (w\j) {};
                };
                \draw [edge0] (v\i) edge (w\i) {};
                \draw [edge1] (v\i) edge (w\j) {};
                \draw [edge1] (v\j) edge (w\i) {};
            };
            
            \draw [edge0] (w0) edge (c) {};
            \draw [edge0] (w1) edge (c) {};
            \draw [edge1] (w2) edge (c) {};
            \draw [edge1] plot[smooth, tension=0.7] coordinates {(w0) (90:0.5) (w1)};
            
            \foreach \i in {0,1,2} {
                \foreach \v in {v,w} {
                    \node [vertex] at (\v\i) {};
                }
            };
            \node [vertex] at (c) {};
        \end{tikzpicture}
        \quad
        \begin{tikzpicture}
            [scale=0.8]
            \tikzset{vertex/.style = {circle, draw, fill=black!50, inner sep=0pt, minimum width=4pt}}
            \tikzset{edge0/.style = {line width=1.2pt, black}}
            \tikzset{edge1/.style = {line width=1pt, blue, opacity=0.7}}
            
            \node [vertex] (v0) at (210:3) {};
            \node [vertex] (v1) at (330:3) {};
            \node [vertex] (v2) at (90:3) {};
            \node [vertex] (w0) at (210:1.5) {};
            \node [vertex] (w1) at (330:1.5) {};
            \node [vertex] (w2) at (90:1.5) {};
            \node [vertex] (c) at (0:0) {};
            
            \foreach \i in {0,1,2} {
                \pgfmathtruncatemacro{\j}{Mod(\i+1,3)}
                \foreach \v in {v,w} {
                    \draw [edge0] (\v\i) edge (\v\j) {};
                };
                \draw [edge0] (v\i) edge (w\i) {};
                \draw [edge1] (v\i) edge (w\j) {};
                \draw [edge1] (v\j) edge (w\i) {};
                \draw [edge0] (w\i) edge (c) {};
            };
        \end{tikzpicture}
        \caption{Simple $1$-planar drawings of~${K_3 + \overline{K_{1,3}}}$, the one on the left is not rich, the one on the right is rich.}
        \label{fig:simplenotrich}
    \end{figure}
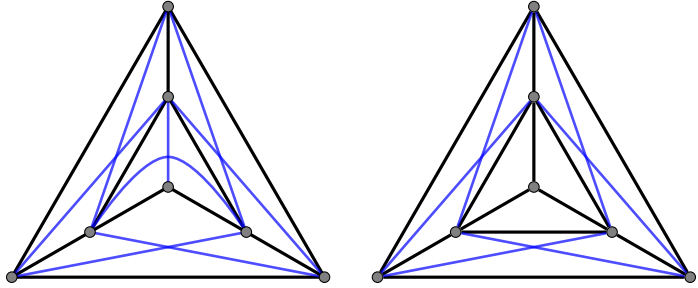

\begin{remark}
    \label{rmk:richrestriction}
    If $\phi$ is a rich $1$-drawing of a graph~$G$ and~${\mathcal{S}(\phi) \subseteq H \subseteq G}$, then~${\phi \restricted H}$ is rich. 
\end{remark}

The following lemma establishes a strong connection between $3$-connected $1$-planar graphs with at least five vertices and the true-planar skeletons of their rich $1$-drawings, in fact allowing us to completely recover the original graph from the true-planar skeleton.

\begin{lemma}
    \label{lem:trueplanarskeletons}
    Let~$\phi$ be a rich $1$-drawing of a $3$-connected edge-maximal $1$-planar graph~$G$ with at least five vertices. 
    The following statements hold.
    \begin{enumerate}
        [label=(\roman*)]
        \item\label{item:trueplanarskeletons1} 
            $\mathcal{S}(\phi)$ is $3$-connected;
        \item\label{item:trueplanarskeletons2} 
            no face of~${\mathcal{S}(\phi)}$ has degree greater than~$4$; 
        \item\label{item:trueplanarskeletons3} 
            the number of faces of~${\mathcal{S}(\phi)}$ of degree~$4$ is equal to the number of crossings of~$\phi$; 
        \item\label{item:trueplanarskeletons4} 
            the number of triangles in~$G$ is at most the number of non-trivial $3$-separators of~${\mathcal{S}(\phi)}$, plus four times the number of 
            faces of~${\mathcal{S}(\phi)}$ of degree~$4$, 
            plus the number of faces of~${\mathcal{S}(\phi)}$ of degree~$3$. 
    \end{enumerate}
\end{lemma}

\begin{proof}
    Let~${H := \mathcal{S}(\phi)}$. 
    For~\ref{item:trueplanarskeletons1}, suppose for a contradiction that there is a non-trivial separation~${(A,B)}$ of~$H$ of order at most~$2$. 
    Since~$G$ is $3$-connected, there is an edge in~${G-E(H)}$ between~${A \setminus B}$ and~${B \setminus A}$. 
    Consider such an edge~$vw$, and let~$xy$ be the edge which crosses~$vw$ in~$\phi$. 
    Since~$\phi$ is rich, ${\{vx,vy,wx,wy\} \subseteq E(H)}$, so~${\{x,y\} = A \cap B}$. 
    In particular, $vw$ is the unique edge crossing~$xy$ in~$\phi$, so~${(A,B)}$ is a separation of~${G-vw}$. 
    Since~$\phi$ is rich, no edge incident with a vertex in~${A \setminus B}$ crosses any edge incident to a vertex in~${B \setminus A}$. 
    Also, since~${xy \notin E(H)}$, no edge incident to~$x$ crosses any edge incident to~$y$. 
    Let~$R$ be a region of~${\mathbb{S}^2 \setminus \phi(G)}$ whose boundary contains~$x$, a segment of an edge~$e$ incident to~$x$ in~${G[A \setminus \{y\}]}$ and a segment of an edge~$e'$ incident to~$x$ in~${G[B \setminus \{y\}]}$. 
    Note that the boundary of~$R$ contains a vertex~$v'$ in~${A \setminus B}$ (either an endvertex of~$e$ or an endvertex of an edge which crosses~$e$). 
    Similarly, the boundary of~$R$ contains a vertex~$w'$ in~${B \setminus A}$. 
    Since~$G$ is edge-maximal, we have~${v'w' \in E(G)\setminus E(H)}$, and so~${v'w'}$ crosses~${xy}$, and hence~${\{v',w'\} = \{v,w\}}$. 
    Now~$x$ and~$y$ lie in different regions of~${\mathbb{S}^2 \setminus (R \cup \phi(vw))}$. 
    Hence, every component of~${G - \{v,w,x,y\}}$ has at most two neighbours in~${\{v,w,x,y\}}$, contradicting that~$G$ is $3$-connected.
    
    For~\ref{item:trueplanarskeletons2}, suppose for a contradiction that some facial cycle~$C$ of~$H$ has length at least~$5$. 
    Since~$K_5$ is not planar, some pair of vertices of~$C$ are not adjacent in~$H$. 
    Since~$G$ is maximal, there is a crossed edge between some pair of vertices on~$C$ which are not adjacent in~$H$. 
    Since~$H$ is $3$-connected, this crossed edge is drawn in the face bounded by~$C$. 
    Hence, since~$\phi$ is rich there is a $4$-cycle in~$H$ containing only vertices of~$C$. 
    Some edge of this $4$-cycle is not in~$C$, and deleting the endvertices of this edge disconnects~$H$, contradicting~\ref{item:trueplanarskeletons1}. 

    For~\ref{item:trueplanarskeletons3}, consider a pair of edges~$e$ and~$f$ which cross in~$\phi$. 
    Since~$\phi$ is rich, these edges have four distinct endvertices in total which, by~\ref{item:trueplanarskeletons2}, induce a facial cycle of~$H$. 
    Now consider a facial cycle~$C$ of~$H$ of length~$4$. 
    Since there is no facial cycle of length~$4$ with respect to any planar drawing of~$K_4$, there is a pair vertices of~$C$ which are not adjacent in~$H$. 
    As~$G$ is maximal, there is a crossed edge between some pair of vertices of~$C$ which are not adjacent in~$H$. 
    Since~$H$ is $3$-connected, this crossed edge is drawn in the face bounded by~$C$. 
    Now, since every face of~$H$ has degree at most~$4$, there is a one-to-one correspondence between crossings of~$\phi$ and faces of~$H$ of degree~$4$. 
    
    For~\ref{item:trueplanarskeletons4}, observe that trivially the number of triangles of~$G$ whose vertices all lie on a facial cycle~$F$ of~$H$ is at most four if~$F$ has length~$4$, 
    and at most one if~$F$ has length~$3$. 
    Consider a triangle~$T$ of $G$ whose vertices do not all lie in any facial cycle of~$H$. 
    If no edge of~$T$ is crossed in~$\phi$, then~$T$ is a non-facial triangle of~$H$, and hence a separating triangle of~$H$. 
    Suppose instead that some edge~$vw$ of~$T$ is crossed by an edge~${v'w'}$. 
    The vertices~$v$, $w$, $v'$ and~$w'$ are all on the facial cycle of~$H$ in which this crossing is drawn, so by assumption neither~$v'$ nor~$w'$ is a vertex of~$T$. 
    Hence, ${\phi(v')}$ and~${\phi(w')}$ are in different regions of~${\mathbb{S}^2 \setminus \phi(T)}$. 
    It follows that~$v'$ and~$w'$ are in separate components of~${H-V(T)}$. 
    Hence, if~$T$ is a triangle of~$G$ 
    such that~${V(T)}$ is not contained in a facial cycle of~$H$, then~${V(T)}$ is a non-trivial $3$-separator of~$H$. 
\end{proof}

\begin{corollary}
    \label{cor:crossingcount}
    A rich $1$-drawing~$\varphi$ of an edge-maximal $3$-connected $1$-planar graph~$G$ on at least~$5$ vertices has exactly~${\abs{E(G)}-3\abs{V(G)}+6}$ crossings. 
\end{corollary}

\begin{proof}
    Let~${H := \mathcal{S}(\phi)}$. 
    By Lemma~\ref{lem:trueplanarskeletons}, $H$ is a $3$-connected planar graph with no face of degree more than~$4$. 
    Hence, $H$ can be extended to a planar triangulation by adding exactly one edge for each facial cycle of length~$4$. 
    The result then follows from Lemma~\ref{lem:trueplanarskeletons}\ref{item:trueplanarskeletons3}. 
\end{proof}

\section{%
\texorpdfstring{Bounding the number of cliques in $1$-planar graphs excluding~$K_6$}%
{Bounding the number of cliques in 1-planar graphs excluding K6}}
\label{sec:nok6}

In this section, we bound the number of cliques of any and all sizes in $1$-planar graphs excluding subgraphs isomorphic to $K_6$. 
The key tool which allows us to do this is the following lemma. 
While it only references triangles, it is easy to bound the number of cliques of any size which contain a given edge in terms of the number of triangles which contain it. 

\begin{lemma}
    \label{lem:4trianglecrossings}
    Let~$\phi$ be a rich $1$-drawing of a graph~$G$ with at least one crossing. 
    If every crossed edge of~$G$ is contained in at least four triangles, then~$G$ has an induced subgraph~$H$ isomorphic to either~$K_6$ or~$K_{2,2,2,2}$. 
\end{lemma}

\begin{proof}
    We first observe the following easy fact.
    \begin{enumerate}
        [label=\upshape{(}$\dagger$\upshape{)}]
        \item\label{item:crossingtriangles} Every $1$-drawing of a graph restricted to the union of two vertex-disjoint triangles has~$0$ or~$2$ crossings.
    \end{enumerate}

    Let~$v_1v_2$ and~$x_1x_2$ be two edges which cross, chosen if possible such that~$v_1$,~$v_2$,~$x_1$ and~$x_2$ have a common neighbour.  
    Since~$\phi$ is rich, ${\{v_1,v_2,x_1,x_2\}}$ induces a~$K_4$ in~$G$.
    Choose distinct common neighbours~$w_1$ and~$w_2$ of~$v_1$ and~$v_2$ in~${G - \{x_1,x_2\}}$, and distinct common neighbours $y_1$ and $y_2$ of $x_1$ and $x_2$ in $G-\{v_1,v_2\}$, so that $\abs{\{w_1,w_2\}\cap \{y_1,y_2\}}$ is maximised.
    
    First, suppose~$\{w_1,w_2\}=\{y_1,y_2\}$. By~\ref{item:crossingtriangles}, for some~${i,j \in [2]}$, the edge~${w_1v_i}$ crosses~$w_2x_j$. 
    Hence, since~$\phi$ is rich,~$w_1$ and~$w_2$ are adjacent. 
    It follows that~${H := G[\{v_1,v_2,w_1,w_2,x_1,x_2\}]}$ is isomorphic to~$K_6$ 
    (see Figure~\ref{fig:4trianglecrossings}(a)). 

    Suppose instead that ${\abs{\{w_1,w_2\} \cap \{y_1,y_2\}} = 1}$. Without loss of generality, we may assume that~${w_1 = y_1}$. 
    By~\ref{item:crossingtriangles}, without loss of generality, $v_2w_1$ crosses $x_1y_2$. 
    Since~$\phi$ is rich, $w_1x_1$ and $v_2x_2$ are edges of $G$ which are uncrossed with respect to $\phi$. Now by~\ref{item:crossingtriangles} applied to $w_2v_1v_2w_2$ and $w_1x_1x_2w_1$, for some $i\in [2]$, the edges $v_iw_2$ and $x_2w_1$ cross with respect to $\phi$. 
    By~\ref{item:crossingtriangles}, there is no triangle containing $x_1y_2$ which is vertex-disjoint from~$w_1v_2x_2w_1$. Hence,~$x_1y_2$ is a crossed edge which is contained in at most three triangles 
    (see Figure~\ref{fig:4trianglecrossings}(b)).
    
    Finally, suppose~${\{w_1,w_2\} \cap \{y_1,y_2\} = \emptyset}$. 
    By~\ref{item:crossingtriangles}, for every~${i,j \in [2]}$ there exist~${k,\ell \in [2]}$ such that~$v_iw_j$ crosses~$x_ky_\ell$. 
    Without loss of generality,~$x_1y_1$ crosses~$v_1w_1$. 
    Since~$\phi$ is rich,~$x_1w_1$,~$y_1v_1$ and~$y_1w_1$ are all uncrossed edges of~$G$.
    Hence, by our choice of~$w_1$,~$w_2$,~$y_1$ and~$y_2$ we have that~$y_1$ is not adjacent to~$v_2$ and that~$w_1$ is not adjacent to~$x_2$. 
    Applying~\ref{item:crossingtriangles} to the triangles~${x_1x_2y_2x_1}$ and~${v_1v_2w_1v_1}$, we find that~$w_1v_2$ crosses~$y_2x_1$, and hence~$v_2y_2$ and~$w_1y_2$ are uncrossed edges of~$G$. 
    Similarly~$w_2v_1$ crosses~$y_1x_2$ and finally~$w_2v_2$ crosses~$y_2x_2$, 
    and so~$w_2y_1$,~$w_2x_2$ and~$w_2y_2$ are all uncrossed edges of~$G$. 
    
    Now consider the pair of crossing edges~$y_1x_1$ and~$v_1w_1$. 
    Let~$w'_1$ and $w'_2$ be the two common neighbours of~$v_1$ and~$w_1$ distinct from~$x_1$ and~$y_1$, and let~$y'_1$ and~$y'_2$ be the two common neighbours of~$x_1$ and~$y_1$ distinct from~$v_1$ and~$w_1$. 
    By our choice of~$v_1v_2$ and~$x_1x_2$, we know that~${\{w'_1,w'_2\} \cap \{y'_1,y'_2\} = \emptyset}$. 
    Applying~\ref{item:crossingtriangles}, there are at least four pairs of crossing edges such that one edge is contained in~${A := \{v_1w'_1,v_1w'_2,w_1w'_1,w_1w'_2\}}$ and the other is contained in~${B := \{x_1y'_1,x_1y'_2,y_1y'_1,y_1y'_2\}}$. 
    Therefore each edge in~$A$ crosses exactly one edge in~$B$, and vice versa. 
    Observe that~${v_2 \in \{w'_1,w'_2\}}$, so~${v_2w_1 \in A}$, which means~${x_1y_2 \in B}$, so~$y_2y_1$ is an edge in~$B$. 
    Similarly,~$w_1w_2$ is an edge in~$A$, and by process of elimination these edges cross. 
    Now~${H := G[\{v_1,v_2,w_1,w_2,x_1,x_2,y_1,y_2\}]}$ is isomorphic to~$K_{2,2,2,2}$  
    (see Figure~\ref{fig:4trianglecrossings}(c)). \qedhere
    
    \begin{figure}[htbp]
        \centering
        \begin{tikzpicture}
            [scale=0.8]
            \tikzset{vertex/.style = {circle, draw, fill=black!50, inner sep=0pt, minimum width=4pt}}
            \tikzset{edge0/.style = {line width=1.2pt, black}}
            \tikzset{edge1/.style = {line width=1pt, blue, opacity=0.7}}
            
            \node [vertex, label=below:$\scriptstyle x_2$] (v0) at (210:3) {};
            \node [vertex, label=below:$\scriptstyle v_2$] (v1) at (330:3) {};
            \node [vertex, label=above:$\scriptstyle w_2$] (v2) at (90:3) {};
            \node [vertex, label=3:$\scriptstyle v_1$] (w0) at (210:1.5) {};
            \node [vertex, label=177:$\scriptstyle x_1$] (w1) at (330:1.5) {};
            \node [vertex, label={[label distance=4pt] 272.5:$\scriptstyle w_1$}] (w2) at (90:1.5) {};
            
            \foreach \i in {0,1,2} {
                \pgfmathtruncatemacro{\j}{Mod(\i+1,3)}
                \foreach \v in {v,w} {
                    \draw [edge0] (\v\i) edge (\v\j) {};
                };
                \draw [edge0] (v\i) edge (w\i) {};
                \draw [edge1] (v\i) edge (w\j) {};
                \draw [edge1] (v\j) edge (w\i) {};
            };
            
            \node at (270:2.5) {(a)};
        \end{tikzpicture}
        \quad
        \begin{tikzpicture}
            [scale=0.8]
            \tikzset{vertex/.style = {circle, draw, fill=black!50, inner sep=0pt, minimum width=4pt}}
            \tikzset{edge0/.style = {line width=1.2pt, black}}
            \tikzset{edge1/.style = {line width=1pt, blue, opacity=0.7}}
            
            \coordinate [label=above:$\scriptstyle{w_1}$] (v0) at (90:3) {};
            \coordinate [label=below:$\scriptstyle w_2$] (v1) at (210:3) {};
            \coordinate [label=below:$\scriptstyle y_2$] (v2) at (330:3) {};
            \coordinate [label=5:$\scriptstyle v_1$] (v3) at (160:0.75) {};
            \coordinate [label=175:$\scriptstyle x_1$] (v4) at (20:0.75) {};
            \coordinate [label=below:$\scriptstyle x_2$] (v5) at (210:1.5) {};
            \coordinate [label=below:$\scriptstyle v_2$] (v6) at (330:1.5) {};
           
            \foreach \i in {1,...,6} {
                \pgfmathtruncatemacro{\p}{Mod(\i,2)}
                \draw [edge\the\numexpr 1-\p \relax] (v\i) edge (v\the\numexpr \i+\p \relax) {};
                   \ifthenelse{\i < 5}{
                       \draw [edge0] (v0) edge (v\i) {};
                       \draw [edge0] (v\i) edge (v\the\numexpr 6-\p \relax) {};
                       \draw [edge1] (v\i) edge (v\the\numexpr 5+\p \relax) {};
                   }{
                       \draw [edge1] (v0) edge (v\i) {};
                       \draw [edge1] (v\the\numexpr \i-4 \relax) edge (v\the\numexpr \i-2 \relax) {};
                   }
            }
            
            \foreach \i in {0,...,6} {
                \node [vertex] at (v\i) {};
            };
            
            \node at (270:2.5) {(b)};
        \end{tikzpicture}
        \quad
        \begin{tikzpicture}
            [scale=0.8]
            \tikzset{vertex/.style = {circle, draw, fill=black!50, inner sep=0pt, minimum width=4pt}}
            \tikzset{edge0/.style = {line width=1.2pt, black}}
            \tikzset{edge1/.style = {line width=1pt, blue, opacity=0.7}}
            
            \node [vertex, label={[label distance=-2pt] above:$\scriptstyle w_1$}] (v0) at (45:2.5) {};
            \node [vertex, label={[label distance=-2pt] above:$\scriptstyle y_1$}] (v1) at (135:2.5) {};
            \node [vertex, label=below:$\scriptstyle w_2$] (v2) at (225:2.5) {};
            \node [vertex, label=below:$\scriptstyle y_2$] (v3) at (315:2.5) {};
            \node [vertex, label=above:$\scriptstyle x_1$] (w0) at (45:1) {};
            \node [vertex, label=above:$\scriptstyle v_1$] (w1) at (135:1) {};
            \node [vertex, label=below:$\scriptstyle x_2$] (w2) at (225:1) {};
            \node [vertex, label=below:$\scriptstyle v_2$] (w3) at (315:1) {};
            
            \foreach \i in {0,1,2,3} {
                \pgfmathtruncatemacro{\j}{Mod(\i+1,4)}
                \pgfmathtruncatemacro{\k}{Mod(\i+2,4)}
                \foreach \v in {v,w} {
                    \draw [edge0] (\v\i) edge (\v\j) {};
                };
                \draw [edge0] (v\i) edge (w\i) {};
                \draw [edge1] (v\i) edge (w\j) {};
                \draw [edge1] (v\j) edge (w\i) {};
                \draw [edge1] (w\i) edge (w\k) {};
            };
            \draw [edge1] plot[smooth, tension=0.7] coordinates {(v0) (135:3) (v2)};
            \draw [edge1] plot[smooth, tension=0.7] coordinates {(v1) (45:3) (v3)};
            
            \node at (270:2.75) {(c)};
        \end{tikzpicture}
        \caption{The three cases in the proof of Lemma~\ref{lem:4trianglecrossings}.}
        \label{fig:4trianglecrossings}
    \end{figure}
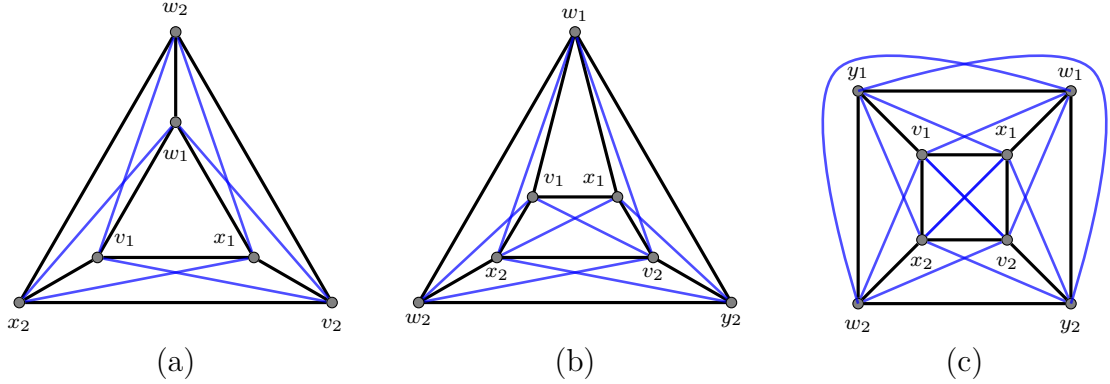
\end{proof}

\begin{lemma}
    \label{lem:K6freebound}
    Let~${n \geq 3}$ be an integer, let~$G$ be a $3$-connected $n$-vertex graph with no subgraph isomorphic to~$K_6$ and let~$\phi$ be a rich $1$-drawing of~$G$ with exactly~$c$ crossings. 
    Then 
    \begin{align*}
        \mathcal{N}(G,K_t) \leq
        \begin{cases}
            3c + 3n - 8 &\textnormal{ if } t = 3,\\
            3c + n - 3  &\textnormal{ if } t = 4,\\
            c           &\textnormal{ if } t = 5. 
        \end{cases}
    \end{align*}
    
    Furthermore, if~${t \in \{3,4\}}$ and equality holds, then~$G$ has exactly~${c + 3n - 6}$ edges. 
\end{lemma}

\begin{proof}
    First suppose that~$G$ contains a subgraph~$H$ isomorphic to~$K_{2,2,2,2}$.
    By Corollary~\ref{cor:drawing-properties}\ref{item:dp1} no edge in~${E(G) \setminus E(H)}$ crosses any edge of~$H$. 
    
    Suppose for contradiction that~${n \geq 9}$, and let~$C$ be a component of~${G - V(H)}$. 
    Note that there is some region~$R$ of~${\mathbb{S}^2 \setminus \phi(H)}$ which contains~$\phi(C)$. 
    Note that every region of~${\mathbb{S}^2 \setminus \phi(H)}$ has exactly two vertices in its boundary by Corollary~\ref{cor:drawing-properties}\ref{item:dp2}. 
    The two vertices in the boundary of~$R$ separate~${V(C)}$ from~${V(H)}$ in~$G$, contradicting the $3$-connectivity of~$G$. 
    Therefore ${n=8}$, and so~${G \cong K_{2,2,2,2}}$ by Theorem~\ref{thm:edgeextremal1planar}. 
    Note that~${\mathcal{N}(G,K_3) = 32}$, ${\mathcal{N}(G,K_4) = 16}$ and~${\mathcal{N}(G,K_5) = 0}$, as required.
    
    Now suppose that~$G$ contains no subgraph isomorphic to~$K_{2,2,2,2}$. Let~${t \in \{3,4,5\}}$, let~${G_0 := G}$ and for each~${i \in [c]}$, let~$G_i$ be a graph obtained from~$G_{i-1}$ by deleting an edge~$e_i$ which is crossed in~${\phi \restricted G_{i-1}}$ 
    and subject to this is contained in as few triangles of $G_{i-1}$ as possible.
    By Remark~\ref{rmk:richrestriction} and Lemma~\ref{lem:4trianglecrossings}, for all~${i \in [c]}$, the edge~$e_i$ is contained in at most three triangles and hence in at most~${\binom{3}{t-2}}$ copies of~$K_t$ in~$G_{i-1}$. 
    It follows that~${\mathcal{N}(G_0,K_t) \leq {\binom{3}{t-2}}c + \mathcal{N}(G_c,K_t)}$. 
    Since~$\phi$ has exactly~$c$ crossings,~$G_c$ is an $n$-vertex planar graph and hence contains no copy of~$K_5$, at most~${3n-8}$ copies of~$K_3$ by Theorem~\ref{thm:planarK3}, and at most~${n-3}$ copies of~$K_4$ by Theorem~\ref{thm:planarK4}. 
    This yields the required bounds for~${\mathcal{N}(G,K_t)}$. 
    Furthermore, if~${t \in \{3,4\}}$ and equality holds, then~$G_c$ is an Apollonian network, and thus the number of edges in~$G$ is~${c+3n-6}$. 
\end{proof}

The following is an immediate corollary of Lemma~\ref{lem:rich1drawing}, Corollary~\ref{cor:crossingcount}, Theorem~\ref{thm:edgeextremal1planar} and Lemma~\ref{lem:K6freebound}. 

\begin{corollary}
    \label{cor:K6freebound}
    If~${n \geq 5}$ and~$G$ is a $3$-connected edge-maximal $1$-planar graph with~$n$ vertices and no subgraph isomorphic to~$K_6$, then
    \begin{align*}
        \mathcal{N}(G,K_t) \leq
        \begin{cases}
            6n - 14 &\textnormal{ if } t = 3,\\
            4n - 9  &\textnormal{ if } t = 4,\\
             n - 2  &\textnormal{ if } t = 5.
        \end{cases}
    \end{align*}
    In particular,~$G$ contains at most~${16n - 32}$ cliques. 
    Furthermore, if~$G$ has less than~${4n-8}$ edges, then~$G$ contains at most~${16n-40}$ cliques. 
    \qed
\end{corollary}

\section{%
\texorpdfstring{Maximising cliques in $1$-planar graphs}%
{Maximising cliques in 1-planar graphs}}
\label{sec:mainproofs}

In this section we complete the proofs of all of our main theorems. The following lemma is an important tool in each of these proofs, allowing us to apply Lemma~\ref{lem:separatingstitch} to extremal graphs which contain subgraphs isomorphic to~$K_6$.

\begin{lemma}
    \label{lem:K6subgraphseparations}
    If~$G$ is a $1$-planar graph with a proper subgraph~$H$ isomorphic to~$K_6$, 
    then either there is a simple $1$-drawing~$\phi$ of~$G$ such that~${\mathcal{S}(\phi)}$ contains a separating triangle of~$G$, 
    or there are at least three components of~${G-V(H)}$ which have at most two neighbours in~$H$.
\end{lemma}

\begin{proof}
    Let~$G'$ be an edge maximal $1$-planar graph containing~$G$ as a subgraph, let~$\phi'$ be a rich $1$-drawing of~$G'$ and let~$\phi_0$ be the restriction of~$\phi'$ to~$G$.  
    By Corollary~\ref{cor:drawing-properties}\ref{item:dp1} no edge in~${E(G) \setminus E(H)}$ crosses any edge of~$H$ in~$\phi_0$. 
    
    Let~$s$ be the number of components of~${G - V(H)}$, and let~${\{ C_i \colon i \in [s]\}}$ be the set of these components. 
    Note that since~$H$ is a proper subgraph of~$G$, we have that~${s \geq 1}$. 
    Now, for~${i \in [s]}$, there is a region~$R$ of~${\mathbb{S}^2 \setminus \phi_0(H)}$ containing~${\phi_0(C_i)}$. 
    The vertices in the boundary of~$R$ separate~${V(C_i)}$ from~${V(H)}$ in~$G$. 
    If the boundary of~$R$ is a face of~${\phi_0 \restricted H}$ of degree~$3$, then its facial triangle is a separating triangle of~$G$ contained in~${\mathcal{S}(\phi_0)}$, as desired. 
    Otherwise, such a component has at most two neighbours in~$H$ by Corollary~\ref{cor:drawing-properties}\ref{item:dp2}, so we may assume that~${s \leq 2}$. 
    For~${i \in [s]}$, let~$S_i$ be the set of neighbours of~$C_i$ in~$H$, and let~${G_i := G[V(C_i) \cup S_i]}$. 
    Now there are at most two edges of~$H$ in ${\bigcup \{ G_i \colon i \in [s]\}}$, so there is a $1$-drawing~$\phi^\ast$ of~$H$ in which each of these edges is in a facial cycle by Corollary~\ref{cor:drawing-properties}\ref{item:dp3} and the symmetry of $K_6$. 
    By considering the restriction of~$\phi_0$ to~${G_i}$, with~$s$ applications of Lemma~\ref{lem:stitching} we obtain a simple $1$-drawing~$\phi$ of~$G$ such that~${\phi \restricted H}$ is equivalent to~$\phi^\ast$. 
    Now some facial cycle with respect to~$\phi^\ast$ contains~$S_1$, and this is a separating triangle of~$G$.
\end{proof}

\subsection{%
\texorpdfstring{Maximising triangles in $1$-planar graphs}%
{Maximising triangles in 1-planar graphs}}

\begin{lemma}
    \label{lem:3connectivity}
    Let~$G$ be a graph with at least four vertices and at least~${(19\abs{V(G)}-62)/3}$ triangles. 
    If~${(A,B)}$ is a separation of~$G$ of order at most~$2$, 
    then for some~${X \in \{A,B\}}$ we have that~${G[X]}$ has at least~${(19\abs{X}-53)/3}$ triangles. 
\end{lemma}

\begin{proof}
    Since~${(A,B)}$ is a separation of order at most~$2$, we have that 
    \[
        \mathcal{N}(G,K_3)=\mathcal{N}(G[A],K_3)+\mathcal{N}(G[B],K_3).
    \]
    Suppose for a contradiction that~$\mathcal{N}(G[X],K_3)<(19|X|-53)/3$ for each~${X \in \{A,B\}}$. 
    Then
    \begin{align*}
        \mathcal{N}(G,K_3)
        &< \frac{19(|A|+|B|)-106}{3}
        \leq \frac{19(|V(G)|+2)-106}{3}
        < \frac{19|V(G)|-62}{3}.\qedhere
    \end{align*}
\end{proof}

\begin{lemma}
    \label{lem:computation8}
    Let~$k$,~$s$ and~$n$ be integers such that~${k \in \{1,2\}}$,~${s \in \{0,1,2\}}$ and~${n = 3 k + s}$, 
    and let~$G$ be a $1$-planar graph with~$n$ vertices. 
    If~$G$ contains at least~${19 k + 5s - 18}$ triangles, 
    then~$G$ is isomorphic to one of~$K_3$,~$K_6$,~${K_3 + C_4}$,~${K_2+\overline{P_6}}$ or~$K_{2,2,2,2}$. 
    In particular,~$G$ contains exactly~${19 k + 5s - 18}$ triangles unless~$G$ is isomorphic to~$K_{2,2,2,2}$. 
\end{lemma}

\begin{proof}
    The cases where~${n \leq 6}$ are trivial, so we assume~${k = 2}$ and~${s \in \{1,2\}}$. 
    If~${s = 1}$, the claim follows easily from Corollary~\ref{cor:7vertex1planar}. 
    If~${s = 2}$, note that no proper subgraph of~${K_2+\overline{P_6}}$ or~$K_{2,2,2,2}$ has at least~$30$ triangles, so we may take~$G$ to be an edge-maximal $1$-planar graph. 
    If~$G$ has a $2$-separation~${(A,B)}$ with~${\abs{A} = 7}$, then~${G[A]}$ has at most~$22$ triangles by Corollary~\ref{cor:7vertex1planar}, and so~$G$ has at most~$23$ triangles.
    If~$G$ has a $2$-separation~${(A,B)}$ with~${\abs{A} \in \{5,6\}}$, then the number of triangles in~$G$ is at most ${\binom{\abs{A}}{3} + \binom{10-\abs{A}}{3}}$, and is therefore at most~$24$. 
    Hence,~$G$ is $3$-connected. 
    Let~$\phi$ be a rich $1$-drawing of~$G$.
    By Lemma~\ref{lem:crossingnumber}, we know that~$\phi$ has at most~$6$ crossings. 
    If~$\phi$ has at most~$4$ crossings, then~$G$ has at most~${3 \cdot 4 + 3 \cdot 8 - 8 = 28}$ triangles by Lemma~\ref{lem:K6freebound}. 

    If~$\phi$ has exactly~$5$ crossings, 
    then by Lemma~\ref{lem:trueplanarskeletons}, 
    ${\mathcal{S}(\phi)}$ is a $3$-connected planar graph with~$8$ vertices with~$5$ faces of degree~$4$, whose other faces each have degree~$3$. Note the number of edges of a $3$-connected planar graph can be determined from the degrees of its faces, and then the number of vertices can be deduced from Euler's formula. 
    Since~$G$ has~$8$ vertices, we can deduce that~${\mathcal{S}(\phi)}$ has exactly two faces of degree~$3$, and exactly~$13$ edges. 
    Furthermore, the number of non-trivial $3$-separators in~${\mathcal{S}(\phi)}$ is at least~${30 - 4 \cdot 5 - 2 = 8}$, by Lemma~\ref{lem:trueplanarskeletons}\ref{item:trueplanarskeletons4}. 
    Using the program \emph{plantri} written by Brinkmann and McKay~\cite{plantri}, we find that there is a unique $3$-connected $8$-vertex planar graph with~$13$ edges, faces of degree at most~$4$ and at least eight non-trivial $3$-separators. 
    The unique maximal $1$-planar graph with this true-planar skeleton is~${K_2+\overline{P_6} }$ (see Figure~\ref{fig:K2+P5C}), since it is formed by adding edges between all pairs of vertices of~${\mathcal{S}(\phi)}$ that share a face. 
    
    If~$\phi$ has exactly~$6$ crossings, then by Lemma~\ref{lem:trueplanarskeletons},~${\mathcal{S}(\phi)}$ is a $3$-connected planar quadrangulation. 
    The cube is the unique $3$-connected planar quadrangulation on $8$-vertices, and~$K_{2,2,2,2}$ is the unique maximal $1$-planar graph with this true-planar skeleton.
\end{proof}

\begin{lemma}
    \label{lem:computation}
    Let~$k$,~$s$ and~$n$ be integers such that~${3 \leq k \leq 6}$,~${s \in \{0,1,2\}}$ and~${n=3k+s}$, 
    and let~$\phi$ be a rich $1$-drawing of an 
    edge-maximal $1$-planar 
    $3$-connected graph with~$n$ vertices. 
    If~$G$ contains at least~${19k + 5s - 18}$ triangles, then~$\mathcal{S}(\phi)$ contains a separating triangle of~$G$. 
\end{lemma}

\begin{proof}
    By Lemma~\ref{lem:K6subgraphseparations}, we may assume~$G$ has no subgraph isomorphic to~$K_6$. 
    Hence, by Lemma~\ref{lem:K6freebound},~$G$ has at most~${3c+3n-8}$ triangles, where~$c$ is the number of crossings of~$\phi$. 
    It follows that~${c \geq n-(10+s-k)/3}$. 
    Note that~${(10+s-k)/3 \leq 3}$ with equality if and only if~${n = 11}$, and that~${c \leq n-2}$ by Lemma~\ref{lem:crossingnumber}. It follows that either $n=11$ and $c=8$, or $c=n-2$. 
    
    Suppose that~${n = 11}$ and~${c = 8}$. 
    Note that the number of triangles in~$G$ is at least~$49$. 
    By Lemma~\ref{lem:trueplanarskeletons},~${\mathcal{S}(\phi)}$ has exactly eight faces of degree~$4$, and the other faces of~${\mathcal{S}(\phi)}$ are of degree~$3$. 
    Thus there are exactly two facial triangles of~${\mathcal{S}(\phi)}$, by Euler's formula.
    By Lemma~\ref{lem:trueplanarskeletons}\ref{item:trueplanarskeletons4}, there are at least~${49 - 4 \cdot 8 - 2 = 15}$ non-trivial $3$-separators of~$\mathcal{S}(\phi)$.
    Using \emph{plantri}~\cite{plantri}, we find that there are only three planar graphs\footnote{These graphs are true-planar skeletons of rich $1$-drawings of a stitch of~$K_6$ and~${K_2 + \overline{P_6}}$ and the two possible stitches of~${K_3 + C_4}$ and itself. }
    satisfying the conditions on~${\mathcal{S}(\phi)}$, each of which contains a separating triangle of~${\mathcal{S}(\phi)}$. 
    Note that since no edge of~${\mathcal{S}(\phi)}$ is crossed with respect to~$\phi$, every separating triangle of~${\mathcal{S}(\phi)}$ is a separating triangle of~$G$. 
    
    Suppose instead that~${c = n-2}$. 
    By Lemma~\ref{lem:trueplanarskeletons},~${\mathcal{S}(\phi)}$ is a $3$-connected planar graph with~${n-2}$ faces of degree~$4$, 
    and therefore a $3$-connected planar quadrangulation. 
    Furthermore, the number of triangles in~$G$ is at most~${4(n-2)}$ plus the number of non-trivial $3$-separators of~${\mathcal{S}(\phi)}$. 
    Again using \emph{plantri} to generate all $3$-connected planar quadrangulations on at most~$20$ vertices, 
    it is quick to computationally verify that no such quadrangulation has the requisite number of non-trivial $3$-separators. 
\end{proof}

\begin{lemma}\label{lem:extremaltriangles}
    Let~$k$,~$s$ and~$n$ be integers such that~${k \geq 1}$, ${s \in \{0,1,2\}}$ and~${n = 3k+s}$, 
    and let~$\phi$ be a rich $1$-drawing of a graph~$G$ with~$n$ vertices and at least~${19k + 5s - 18}$ triangles. 
    Then~$G$ either contains exactly~${19k + 5s - 18}$ triangles or is isomorphic to~$K_{2,2,2,2}$. 
    
    Furthermore, if~${k \geq 3}$, then for each~${i \in [2]}$ there exist 
    integers~${k_i \geq 2}$ and~${s_i \in \{0,1,2\}}$ with~${s_1 + s_2 \leq 2}$
    and a ${(3k_i + s_i)}$-vertex graph~$H_i$ with exactly~${19k_i + 5s_i - 18}$ triangles such that~$G$ is a stitch of~$H_1$ and~$H_2$. 
\end{lemma}

\begin{proof}
    We proceed by induction on~$n$. 
    For~${n\leq 8}$ the result follows from Lemma~\ref{lem:computation8}, so we assume~${n \geq 9}$.

    We first show that~$G$ is $3$-connected. 
    Suppose for a contradiction that~$G$ has a non-trivial separation~${(A,B)}$ of order at most~$2$. 
    By Lemma~\ref{lem:3connectivity}, we may assume without loss of generality that~${G[A]}$ has at least~${(19 \abs{A}-53)/3}$ triangles. 
    Let~$k'$ and~$s'$ be integers such that~${\abs{A} = 3k'+s'}$ and~${s' \in \{0,1,2\}}$. 
    Note that~${19k'+5s'-18}$ is strictly less than~${(19\abs{A}-53)/3}$, and~$32$ (the number of triangles in~$K_{2,2,2,2}$) is strictly less than~${(19 \cdot 8 - 53) / 3 = 33}$. 
    Hence, by induction,~${\abs{A} \leq 2}$. 
    It follows that the number of triangles in~$G$ equals the number of triangles in~${G[B]}$. 
    However,~${5 \leq \abs{B} \leq n-1}$, so the number of triangles in~${G[B]}$ is less than 
    \[
        (19|B|-53)/3 \leq \frac{19(3k+s-1)-53}{3} < 19k+5s-18,
    \]
    a contradiction. 
    Hence,~$G$ is $3$-connected. 

    If no subgraph of~$G$ is isomorphic to~$K_6$, then~${19k + 5s - 18 \leq 6n - 14}$ by Lemmas~\ref{lem:K6freebound} and~\ref{lem:crossingnumber}. 
    It follows that~${k \leq 6}$, and so there is a separating triangle~$T$ of~$G$ contained in~${\mathcal{S}(\phi)}$ by Lemma~\ref{lem:computation}. 
    If~$G$ contains a subgraph isomorphic to~$K_6$, then~$\mathcal{S}(\phi)$ contains a separating triangle~$T$ of $G$ by Lemma~\ref{lem:K6subgraphseparations}. 
    
    By Lemma~\ref{lem:separatingstitch}, there are graphs~$H_1$ and~$H_2$ such that~$G$ is a stitch of~$H_1$ and~$H_2$. 
    The number of triangles in~$G$ is equal to the number of triangles in~$H_1$ plus the number of triangles in~$H_2$ minus one. 
    For~${i \in [2]}$, let~${k_i \geq 1}$ and~${s_i \in \{0,1,2\}}$ be integers such that~${\abs{V(H_i)} = 3k_i + s_i}$. 
    
    Since~$T$ is a facial triangle with respect to~${\phi \restricted H_i}$, by Corollary~\ref{cor:drawing-properties}\ref{item:dp2} 
    we have that~$H_i$ is not isomorphic to~$K_{2,2,2,2}$. 
    It follows that the number of triangles in~$H_i$ is at most~${19k_i+5s_i-18}$, by induction. 
    Note that either~${s_1 + s_2 = s}$ and~${k_1 + k_2 = k+1}$, or~${s_1 + s_2 = s + 3}$ and ${k_1 + k_2 = k}$. 
    However, in the latter case, the number of triangles in~$G$ is at most
    \[
        19(k_1+k_2) + 5(s_1+s_2) - 37=19k+5s-22 < 19k + 5s - 18.
    \]  
    Hence~${s = s_1 + s_2}$ and~${k_1+k_2 = k+1}$, and the number of triangles in~$G$ is given by
    \[
        \mathcal{N}(H_1,K_3)+\mathcal{N}(H_2,K_3)-1\leq 19(k_1+k_2)+5(s_1+s_2)-37=19k+5s-18. 
    \]
    Since equality holds, for each~${i \in [2]}$ applying Lemma~\ref{lem:computation8} to~$H_i$ implies that~${k_i \geq 2}$. 
\end{proof}

The following theorem immediately implies Theorem~\ref{thm:mainintro1}. 
\begin{thm}
    \label{thm:1planartriangles}
    Let~$n$,~$k$ and~$s$ be non-negative integers with~${s \in \{0,1,2\}}$ and~${n = 3k+s}$.
    An $n$-vertex $1$-planar graph~$G$ contains the extremal number~${f_3(n)}$ of triangles if and only if~${G \in \mathcal{E}_3}$. 
    In particular, 
    \[
        f_3(n) = 
        \begin{cases}
            \binom{n}{3} & \textnormal{ if } n \leq 6,\\
            32 & \textnormal{ if } n = 8,\\
            19k + 5s - 18 & \textnormal{ otherwise. }
        \end{cases}
    \]
\end{thm}

\begin{proof}
    We proceed by induction on~$n$. 
    The result is trivial when~${n \leq 6}$, and follows from Lemma~\ref{lem:computation8} when~${n \in \{7,8\}}$, so we may assume that~${n \geq 9}$. 
    
    Consider an $n$-vertex graph~$G$ in~$\mathcal{E}_3$. 
    By definition, for each~${i \in [2]}$ there are integers~${k_i \geq 2}$ and~${s_i \in \{0,1,2\}}$ with~${s_1 + s_2 \leq 2}$ and a graph~${H_i \in \mathcal{E}_3 \cup \{K_2+\overline{P_6}\}}$ with~${\abs{V(H_i)} = 3k_i + s_i}$ such that~$G$ is a stitch of~$H_1$ and~$H_2$. 
    Since~${s_1 + s_2 \leq 2}$ and~${\abs{V(G)} = \abs{V(H_1)} + \abs{V(H_2)} - 3}$, we have~${k_1 + k_2 - 1 = k}$ and~${s_1 + s_2 = s}$.  
    Now~${\mathcal{N}(G,K_3) = \mathcal{N}(H_1,K_3) + \mathcal{N}(H_2,K_3) - 1}$. 
    Note that~${K_2 + \overline{P_6}}$ has~${19 \cdot 2 + 5 \cdot 2 - 18}$ triangles, and that neither~$H_1$ nor~$H_2$ is isomorphic to~$K_{2,2,2,2}$ by Corollary~\ref{cor:drawing-properties}\ref{item:dp2}. 
    By induction, we have
    \[
        {\mathcal{N}(G,K_3) 
        = 19 (k_1 + k_2) + 5 (s_1 + s_2) - 37 
        = 19k+5s-18}.
    \]
    
    Now suppose that~$G$ is an $n$-vertex graph containing~${f_3(n)}$ triangles. 
    To complete the proof, we will show that~${G \in \mathcal{E}_3}$. 
    Note that~${\mathcal{N}(G,K_3) \geq 19k + 5s - 18}$ by the previous argument. 
    By Lemma~\ref{lem:extremaltriangles}, for each~${i \in [2]}$ there exist integers~${k_i \geq 2}$ and~${s_i \in \{0,1,2\}}$ with~${s_1 + s_2 \leq 2}$ and a ${(3k_i+s_i)}$-vertex graph~$H_i$ with exactly~${19k_i + 5s_i - 18}$ triangles 
    such that~$G$ is a stitch of~$H_1$ and~$H_2$. 
    By the induction hypothesis and Lemma~\ref{lem:computation8}, we have that~$k_1$ and~$k_2$ are each at least~$2$ and~$H_1$ and~$H_2$ are isomorphic to graphs in~${\mathcal{E}_3 \cup \{ K_2 + \overline{P_6} \}}$, 
    and hence~${G \in \mathcal{E}_3}$, as desired. 
\end{proof}

\subsection{%
\texorpdfstring{Maximising larger cliques in $1$-planar graphs}%
{Maximising larger cliques in 1-planar graphs}}

\begin{lemma}\label{lem:few2seps}
    Let~$k$,~$s$ and~$t$ be integers with~${k \geq 1}$, ${s \in \{0,1,2\}}$ and~${t \in \{4,5,6\}}$ and~$G$ be a graph with~${3k+s}$ vertices, and for~${i \in [3]}$ let~${(A_i,B_i)}$ be non-trivial separations of~$G$ of order at most~$2$ such that the sets~${A_i \setminus B_i}$ are pairwise disjoint and~${\abs{\bigcap_{i \in [3]} B_i} \geq 3}$.
    If there do not exist integers~${k' \geq 1}$ and~${s' \in \{0,1,2\}}$ such that some induced subgraph of~$G$ with~${3k' + s'}$ vertices contains more than~${(k'-1)\binom{6}{t} + \binom{s'+3}{t}}$ subgraphs isomorphic to~$K_t$, then $G$ contains fewer than~${(k-1)\binom{6}{t}}$ subgraphs isomorphic to~$K_t$. 
\end{lemma}

\begin{proof}
    For each~${i \in [3]}$, we may assume that~${(A_i,B_i)}$ has order exactly~$2$. 
    Let~${G_i := G[A_i]}$, 
    let~${G_4 := G\big[\bigcap_{i \in [3]} B_i \big]}$, and for~${i \in [4]}$ let~$k_i$ and~$s_i$ be integers with~${k_i \geq 1}$ and~${s_i \in \{0,1,2\}}$ such that~${\abs{V(G_i)} = 3k_i + s_i}$. 
    Let~$k^\ast$ and~$s^\ast$ be integers with~${k^\ast \geq 1}$ and~${s^\ast \in \{0,1,2\}}$ such that~${\sum_{i \in [4]} s_i = 3k^\ast + s^\ast}$ 
    and note that since~${3k+s-6 = \sum_{i \in [4]} (3k_i+s_i)}$, we have that~${\sum_{i \in [4]} k_i = k+2-k^\ast}$. 
    Therefore, 
    \begingroup
    \allowdisplaybreaks
    \begin{align*}
        \mathcal{N}(G,K_t)
        &\leq \sum_{i \in [4]} \left( (k_i-1) \binom{6}{t} + \binom{s_i+3}{t} \right)\\
        &= (k-2-k^\ast) \binom{6}{t} + \sum_{i \in [4]} \binom{s_i+3}{t}\\
        &\leq (k-2-k^\ast) \binom{6}{t} + \sum_{i\in [4]} \left( \frac{s_i}{2} \binom{5}{t} \right)\\
        &< (k-2-k^\ast) \binom{6}{t} + \frac{3k^\ast+3}{2} \binom{5}{t}\\
        &= (k-1)\binom{6}{t}+(k^\ast+1)\left( \frac{3}{2}\binom{5}{t}-\binom{6}{t}\right)\\
        &< (k-1) \binom{6}{t}. \qedhere
    \end{align*}
    \endgroup
\end{proof}

\begin{lemma}\label{lem:3sepsKt}
    Let~$k$,~$s$ and~$t$ be integers with~${k \geq 1}$, ${s \in \{0,1,2\}}$ and~${t \in \{4,5,6\}}$, let~$G$ be a graph with~${3k+s}$ vertices, and let~${(A,B)}$ be a non-trivial separation of order~$3$ of~$G$. 
    For~${X \in \{A,B\}}$ let~$k_X$ and~$s_X$ be integers with~${k_X \geq 1}$ and~${s_X \in \{0,1,2\}}$ such that~${\abs{X} = 3k_X + s_X}$.
    If~$G$ contains at least~${(k-1) \binom{6}{t} + \binom{s+3}{t}}$ subgraphs isomorphic to~$K_t$, then one of the following holds
    \begin{enumerate}
        [label=(\arabic*)]
        \item\label{item:3sepsKt-1} for each~${X \in \{A,B\}}$, the graph~${G[X]}$ contains exactly~${(k_X-1) \binom{6}{t} + \binom{s_X+3}{t}}$ subgraphs isomorphic to~$K_t$, and either~${0 \in \{s_A,s_B\}}$, or both~${t = 6}$ and~${s_A = s_B = 1}$;
        \item\label{item:3sepsKt-2} for some~${X \in \{A,B\}}$, the graph~${G[X]}$ contains more than~${(k_X-1) \binom{6}{t} + \binom{s_X+3}{t}}$ subgraphs isomorphic to~$K_t$. 
    \end{enumerate}
\end{lemma}

\begin{proof}
    Assume that~\ref{item:3sepsKt-2} fails. 
    Observe that~${\mathcal{N}(G,K_t) = \mathcal{N}(G[A],K_t) + \mathcal{N}(G[B],K_t)}$. 
    Assume for a contradiction that~${s_A + s_B \geq 3}$, and hence~${k_A + k_B = k}$. 
    Now, since~${s_A, s_B \leq 2}$ and~${t \in \{4,5,6\}}$, we have
    \begin{align*}
        \mathcal{N}(G[A],K_t) + \mathcal{N}(G[B],K_t) 
        &\leq (k_A+k_B-2)\binom{6}{t} + \binom{s_A+3}{t} + \binom{s_B+3}{t}\\ 
        &< (k-2)\binom{6}{t} + \binom{6}{t} + \binom{s+3}{t}
        \leq \mathcal{N}(G,K_t), 
    \end{align*}
    a contradiction. 
    So~${s_A + s_B < 3}$, and hence~${k_A + k_B = k+1}$ and~${s_A + s_B = s}$. 
    Since~${t \geq 4}$,

    \begin{align*}
        \mathcal{N}(G[A],K_t) + \mathcal{N}(G[B],K_t) 
        &\leq (k_A+k_B-2) \binom{6}{t} + \binom{s_A+3}{t} + \binom{s_B+3}{t}\\
        &\leq (k-1) \binom{6}{t} + \binom{s_A+s_B+3}{t} 
        \leq \mathcal{N}(G,K_t). 
    \end{align*}
    Since equality holds, we deduce that~\ref{item:3sepsKt-1} holds. 
\end{proof}

The following theorem immediately implies Theorem~\ref{thm:mainintro2}. 
\begin{thm}
    \label{thm:1planarbigcliques}
    Let~$n$,~$k$ and~$s$ be non-negative integers with~${s \in \{0,1,2\}}$ and~${n = 3k+s}$ and let~${t \in \{4,5,6\}}$. 
    An $n$-vertex $1$-planar graph~$G$ contains the extremal number~${f_t(n)}$ of subgraphs isomorphic to~$K_t$ if and only if~${G \in \mathcal{E}_t}$. 
    In particular, if~${n \geq 3}$ then 
    \[
        f_t(n) = (k-1) \binom{6}{t} + \binom{s+3}{t}. 
    \]
\end{thm}

\begin{proof}
    We proceed by induction on~$n$. 
    The claim is trivial when~${n \leq 6}$, so assume~${n \geq 7}$. 
    
    Consider an $n$-vertex graph~$G$ in~$\mathcal{E}_t$. 
    If~${G \cong K_3+C_4}$, then~${t = 4}$ and~${\mathcal{N}(G,K_4) = 16}$, as required. 
    Otherwise, for each~${i \in [2]}$ there exist integers~${k_i \geq 1}$ and~${s_i \in \{0,1,2\}}$ with~${s_1 + s_2 \leq 2}$ and a graph~${H_i \in \mathcal{E}_t}$ with~${\abs{V(H_i)} = 3k_i + s_i}$ such that~$G$ is a stitch of~$H_1$ and~$H_2$, with either~${s_1 = 0}$ or~${t = 6}$. 
    Since~${s_1 + s_2 \leq 2}$ and~${\abs{V(G)} = \abs{V(H_1)} + \abs{V(H_2)} - 3}$, we have that~${k_1 + k_2 - 1 = k}$ and~${s_1 + s_2 = s}$.  
    Now~${\mathcal{N}(G,K_t) = \mathcal{N}(H_1,K_t) + \mathcal{N}(H_2,K_t)}$ since~${t > 3}$. 
    Note that ${\binom{s_1+3}{t} = 0}$,~${\binom{s_2+3}{6} = 0}$ and~${\binom{s_2+3}{t} =\binom{s+3}{t}}$.
    By induction, we have 
    \begin{align*}
        \mathcal{N}(G,K_t) 
        &= (k_1 - 1) \binom{6}{t} + \binom{s_1+3}{t} + (k_2-1) \binom{6}{t} + \binom{s_2+3}{t}\\
        &= (k-1) \binom{6}{t} + \binom{s+3}{t}. 
    \end{align*}
    
    Now suppose that~$G$ is an $n$-vertex graph containing~${f_t(n)}$ subgraphs isomorphic to~$K_t$. 
    To complete the proof, we will show that~${G \in \mathcal{E}_t}$. 
    By the previous argument we have that~${\mathcal{N}(G,K_t) \geq (k-1) \binom{6}{t} + \binom{s+3}{t}}$. 
    
    First, assume that~$G$ is $4$-connected. 
    By Lemma~\ref{lem:K6subgraphseparations}, $G$ has no subgraph isomorphic to~$K_6$. 
    This is a contradiction if~${t = 6}$. 
    By Corollary~\ref{cor:K6freebound}, we deduce that~${t = 4}$ and~${n = 7}$. 
    By Corollary~\ref{cor:7vertex1planar}, $G$ is isomorphic to a subgraph of~${K_3 + \overline{K_{1,3}}}$ or of~${K_3 + C_4}$. 
    Every proper subgraph of each of these contains strictly fewer subgraphs isomorphic to~$K_4$, so~${G \cong K_3 + C_4}$, as desired. 
    
    Suppose instead that~$G$ is not $4$-connected, and consider a non-trivial $3$-separation of~${(A,B)}$ of~$G$, not necessarily minimal. 
    By Lemma~\ref{lem:3sepsKt} and the induction hypothesis, either~${t = 6}$ or one of~$\abs{A}$, $\abs{B}$ is divisible by~$3$. 
    Hence~$G$ has a subgraph isomorphic to~$K_6$, either because~${(k-1) \binom{6}{6} + \binom{s+3}{6} \geq 1}$ or by Definition~\ref{def:extremalgraphs} and the induction hypothesis. 
    Now by Lemma~\ref{lem:few2seps} and the induction hypothesis, we obtain from Lemma~\ref{lem:K6subgraphseparations} a simple $1$-drawing~$\phi$ of~$G$ such that~${\mathcal{S}(\phi)}$ contains a separating triangle of~$G$. 
    By Lemma~\ref{lem:separatingstitch}, there are graphs~$H_1$ and~$H_2$ such that~$G$ is a stitch of~$H_1$ and~$H_2$. 
    By the induction hypothesis and Lemma~\ref{lem:3sepsKt}, it follows that~${G \in \mathcal{E}_t}$, as desired. 
\end{proof}

\subsection{%
\texorpdfstring{Maximising all cliques in $1$-planar graphs}%
{Maximising all cliques in 1-planar graphs}}

For a graph~$G$, we write~${\mathcal{N}_{\textnormal{cliques}}(G)}$ for the number of subgraphs~${H \subseteq G}$ that are cliques. 

\begin{lemma}\label{lem:3sepsCliques}
    Let~$k$ and~$s$ be integers with~${k \geq 1}$ and ${s \in \{0,1,2\}}$, let~$G$ be a graph with~${3k+s}$ vertices, and let~${(A,B)}$ be a non-trivial separation of~$G$ of order at most~$3$. 
    For~${X \in \{A,B\}}$ let~$k_X$ and~$s_X$ be integers with~${k_X \geq 1}$ and~${s_X \in \{0,1,2\}}$ such that ${\abs{X} = 3k_X + s_X}$.
    If~$G$ contains at least~${56(k-1) + 2^{s+3}}$ cliques, then one of the following statements holds.
    \begin{enumerate}
        [label=(\roman*)]
        \item\label{item:3sepsCliques-1} 
            The order of~$(A,B)$ is exactly~$3$, 
            one of~$s_A$ or~$s_B$ is equal to~$0$, 
            and for all~${X \in \{A,B\}}$, the graph~${G[X]}$ contains exactly~${56(k_X-1)+2^{s_X+3}}$ cliques; 
        \item\label{item:3sepsCliques-2} the order of~$(A,B)$ is exactly~$3$ and $G[A\cap B]$ is not isomorphic to a triangle;
        \item\label{item:3sepsCliques-3} for some~${X \in \{A,B\}}$, the graph~${G[X]}$ contains more than~${56(k_X-1) + 2^{s_X+3}}$ cliques. 
    \end{enumerate}
\end{lemma}

\begin{proof}
    Assume that~\ref{item:3sepsCliques-3} fails. 
    Let~$k^\ast$ and~$s^\ast$ be integers with~${k^\ast \geq -1}$ and~${s^\ast \in \{0,1,2\}}$ such that ${s_A + s_B - \abs{A \cap B} = 3k^\ast + s^\ast}$ 
    and note that we have~${k_A + k_B = k - k^\ast}$ and~${s = s^\ast}$ since~${3k+s = 3(k_A + k_B) + s_A + s_B - \abs{A \cap B}}$. 
    Observe that for all~${s' \in \{0,1,2\}}$, we have~${2^{s'+3} \in \{12s'+4, 12s'+8\}}$. 
    Now
    \begin{align*}
        \mathcal{N}_{\textnormal{cliques}}(G) 
        &= \mathcal{N}_{\textnormal{cliques}}(G[A]) + \mathcal{N}_{\textnormal{cliques}}(G[B]) - \mathcal{N}_{\textnormal{cliques}}(G[A\cap B])\\
        &\leq 56(k_A+k_B-2)+2^{s_A+3}+2^{s_B+3}-\abs{A\cap B}\\
        &\leq 56(k-k^\ast-2)+12(s_A+s_B)+16-\abs{A\cap B}\\
        &\leq 56(k-k^\ast-2)+12(3k^\ast+s+\abs{A\cap B})+16-\abs{A\cap B}\\
        &\leq 56(k-1)+12s-20k^\ast+11\abs{A\cap B}-40.
    \end{align*}
    Since~${\mathcal{N}_{\textnormal{cliques}}(G) \geq  56(k-1)+2^{s+3}}$, we deduce that~${-20k^\ast + 11\abs{A\cap B} - 40 \geq 4}$, and so~${\abs{A \cap B} = 3}$ and~${k^\ast = -1}$. 
    This implies that~${s_A + s_B = s}$, and either~${s_A = s_B = 1}$, or one of~$s_A$ or~$s_B$ is equal to~$0$. 
    We may further assume that~$\mathcal{N}_{\textnormal{cliques}}(G[A\cap B])=8$, or else~\ref{item:3sepsCliques-2} holds. 
    Now
    \begin{align*}
        \mathcal{N}_{\textnormal{cliques}}(G[A]) + \mathcal{N}_{\textnormal{cliques}}(G[B]) - 8
        &\leq 56(k_A+k_B-2)+2^{s_A+3}+2^{s_B+3} - 8\\
        &\leq 56(k-1)+2^{s+3}
        \leq \mathcal{N}_{\textnormal{cliques}}(G)
    \end{align*}
    and since equality holds, so does~\ref{item:3sepsCliques-1}.
\end{proof}

The following theorem immediately implies Theorem~\ref{thm:mainintro3}. 

\begin{thm}\label{thm:1planarallcliques}
    Let~$n$,~$k$ and~$s$ be non-negative integers with~${s \in \{0,1,2\}}$ and~${n = 3k+s}$.
    An $n$-vertex $1$-planar graph~$G$ contains the extremal number~${f(n)}$ of cliques 
    if and only if~${G \in \mathcal{E}}$. 
    In particular, if~${n \geq 3}$ then
    \[
        f(n) = 
            56(k-1)+2^{s+3}.
    \]
\end{thm}
\begin{proof}
    We proceed by induction on~$n$. 
    The claim is trivial when~${n \leq 6}$, so assume~${n \geq 7}$. 
    
    Consider an $n$-vertex graph~$G$ in~$\mathcal{E}$. 
    If~${G \cong K_3+C_4}$, then~${\mathcal{N}_{\textnormal{cliques}}(G) = 72}$, as required.
    Otherwise, for each~${i \in [2]}$ there exist integers~${k_i \geq 1}$ and~${s_i \in \{0,1,2\}}$ and a graph~${H_i \in \mathcal{E}}$ with~${\abs{V(H_i)} = 3k_i + s_i}$ such that~$G$ is a stitch of~$H_1$ and~$H_2$, with~${s_1 = 0}$.
    Since~${\abs{V(G)} = \abs{V(H_1)} + \abs{V(H_2)} - 3}$ and~$s_1=0$, we have that~${k_1 + k_2 - 1 = k}$ and~${s_2 = s}$.  
    Now~${\mathcal{N}_{\textnormal{cliques}}(G) = \mathcal{N}_{\textnormal{cliques}}(H_1) + \mathcal{N}_{\textnormal{cliques}}(H_2) - \mathcal{N}_{\textnormal{cliques}}(K_3)}$. 
    By induction, we have 
    \begin{align*}
        \mathcal{N}_{\textnormal{cliques}}(G) 
        &= 56(k_1 - 1) + 2^{3} + 56(k_2-1) + 2^{s_2+3} - 8\\
        &= 56(k-1) + 2^{s+3}. 
    \end{align*}
    
    Now suppose that~$G$ is an $n$-vertex graph containing~${f(n)}$ cliques. 
    To complete the proof, we will show that~${G \in \mathcal{E}}$. 
    Note that~${\mathcal{N}_{\textnormal{cliques}}(G) \geq 56(k-1) + 2^{s+3}}$ by the previous argument, and that~$G$ is an edge maximal~$1$-planar graph since adding edges to a graph increases the number of cliques. 
    Also~$G$ is~$3$-connected by~Lemma~\ref{lem:3sepsCliques} and the induction hypothesis.
    
    First, assume that~$G$ contains no subgraph isomorphic to~$K_6$.  
    By Corollary~\ref{cor:K6freebound}, we deduce that~${n \in \{7,8\}}$. 
    If~${n = 7}$, then by Corollary~\ref{cor:7vertex1planar}, $G$ is either~${K_3 + \overline{K_{1,3}}}$ or~${K_3 + C_4}$, as desired. 
    Now suppose for a contradiction that~${n = 8}$. 
    First, assume that~$G$ has a vertex~$v$ of degree at most~$4$. 
    Now~${\mathcal{N}_{\textnormal{cliques}}(G) \leq \mathcal{N}_{\textnormal{cliques}}(G-v) + 2^4 \leq 88}$ by induction, and since equality holds we have that~${G - v \cong K_3+C_4}$ 
    and that the degree of~$v$ is exactly~$4$. 
    Since~$G$ is edge maximal, there is a rich $1$-drawing~$\phi$ of~$G$.
    By Corollary~\ref{cor:drawing-properties}\ref{item:dp1} no edge of~${G - v}$ is crossed with respect to~$\phi$, so at least four vertices are contained in the boundary of the region of $\mathbb{S}^2\setminus \phi(G-v)$ containing $\phi(v)$. This contradicts Corollary~\ref{cor:drawing-properties}\ref{item:dp2}.
    We therefore can assume that~${\delta(G) \geq 5}$ and hence~${\Delta(\overline{G}) \leq 2}$. 
    It was shown by Bodendiek, Schuhmacher and Wagner~\cite{BodendiekSW84} that the graph~$K_{2,2,2,2}$ is up to isomorphism the only $8$-vertex $1$-planar graph with~$24$ edges, and since it contains exactly~$81$ cliques, we may assume that~$\overline{G}$ has at least~$5$ edges. 
    It is straightforward to verify that of the nine $8$-vertex graphs with five edges and maximum degree at most~$2$, only~${\overline{C_5 + K_3}}$ contains~$88$ independent sets of vertices. 
    However,~${C_5 + K_3}$ contains~$31$ triangles, and hence is not $1$-planar by Lemma~\ref{lem:computation8}. 
    We therefore deduce that~${n \neq 8}$, as required. 
    
    So we may assume that~$G$ contains a subgraph~$H$ isomorphic to~$K_6$. 
    By Lemma~\ref{lem:K6subgraphseparations}, Lemma~\ref{lem:separatingstitch} and the fact that~$G$ is~$3$-connected, there are graphs~$H_1$ and~$H_2$ such that~$G$ is a stitch of~$H_1$ and~$H_2$. 
    By Lemma~\ref{lem:3sepsCliques} and the induction hypothesis, we deduce that~${G \in \mathcal{E}}$, as required. 
\end{proof}

\section{Generating extremal graphs efficiently with tree-decompositions}
\label{sec:treedec}
In Section~\ref{sec:stitchings}, we gave a recursive definition of the classes of graphs which are extremal in the sense of Theorems~\ref{thm:1planartriangles}, \ref{thm:1planarbigcliques} and~\ref{thm:1planarallcliques}. 
While this recursive definition was useful for the proofs in Section~\ref{sec:mainproofs}, it does not by itself provide a way to generate all of the graphs in these classes. 
The reason for this is that in order to determine which graphs are stitches of two extremal graphs~$H_1$ and~$H_2$, we need to determine which triangles of~$H_1$ and~$H_2$ are facial with respect to some $1$-drawing. 
We solve this problem in Subsection~\ref{subsec:treedec}, which allows us to provide a structural characterisation of the classes of extremal graphs in terms of tree-decompositions.
In Subsection~\ref{subsec:efficent}, we discuss how to use these characterisations to generate the $n$-vertex graphs in each of these classes in polynomial time.

\subsection{A structural characterisation of the extremal graphs in terms of tree-decompositions}
\label{subsec:treedec}

A pair~${(T,\mathcal{V})}$ is a \emph{tree-decomposition of~$G$} if~$T$ is a tree and~$\mathcal{V}$ is a family~${( V_t \colon t \in V(T))}$ of sets of vertices of~$G$ such that
\begin{enumerate}
    [label=(T\arabic*)]
    \item \label{item:T1} ${V(G) = \bigcup \mathcal{V}}$; 
    \item \label{item:T2} for every edge~${e \in E(G)}$ there is a~${t \in V(T)}$ such that~${e \in E(G[V_t])}$; and
    \item \label{item:T3} ${V_{s} \cap V_{t} \subseteq V_{u}}$ whenever~$u$ is a node on the unique path from~$s$ to~$t$ in~$T$.
\end{enumerate}

We call the sets in~$\mathcal{V}$ the \emph{parts} of~$(T,\mathcal{V})$. 
We call the sets~$V_{s} \cap V_{t}$ for each edge~${s t \in E(T)}$ the \emph{adhesion sets} of~$(T,\mathcal{V})$. We call $V(T)$ the set of \emph{nodes} of the tree.

\begin{definition}
    We define sets~$\mathcal{A}(\mathcal{E}_3)$, $\mathcal{A}(\mathcal{E}_4)$, $\mathcal{A}(\mathcal{E}_5)$, $\mathcal{A}(\mathcal{E}_6)$ and $\mathcal{A}(\mathcal{E})$ as follows: 
    
    $\mathcal{A}(\mathcal{E}_3) := \{ K_6, K_3 + C_4, K_2 + \overline{P_6} \}$;
    
    $\mathcal{A}(\mathcal{E}_4) := \{ K_4, K_5, K_6, K_3 + C_4 \}$;
    
    $\mathcal{A}(\mathcal{E}_5) := \{ K_4, K_5, K_6 \}$;
    
    $\mathcal{A}(\mathcal{E}_6) :=\{K_3,K_6\}\cup \{ H \subseteq K_5 \colon  H \textnormal{ is } 3\textnormal{-connected} \}$;
    
    $\mathcal{A}(\mathcal{E})
    := \{ K_4, K_5, K_6, K_3 + C_4 \}$.
\end{definition}

\begin{lemma}
    \label{lem:td1}
    Let~$\mathcal{C}$ be one of $\mathcal{E}_3$, $\mathcal{E}_4$, $\mathcal{E}_5$, $\mathcal{E}_6$ or~$\mathcal{E}$. 
    Every $3$-connected graph~${G \in \mathcal{C}}$ with at least $9$ vertices has a tree-decomposition ${(T,\mathcal{V})}$ with the following properties.
    \begin{enumerate}
        [label=(\alph*), series=tdenum]
        \item \label{item:td1-1} For each node~${t \in V(T)}$ the graph~$G[V_t]$ is isomorphic to a graph in~${\mathcal{A}(\mathcal{C}) \setminus \{ K_3 \}}$;  
        \item \label{item:td1-2} each graph in~${\mathcal{H} := \{ G[V_s \cap V_t] \colon st \in E(T) \}}$ is a triangle;
        \item \label{item:td1-3} one of the following holds:
        \begin{itemize}
            \item at most one part induces a graph not isomorphic to~$K_6$;
            \item $\mathcal{C}$ is~$\mathcal{E}_3$ and two parts each induce a subgraph isomorphic to~${K_3 + C_4}$ and all other parts induce a subgraph isomorphic to~$K_6$;
            \item $\mathcal{C}$ is~$\mathcal{E}_6$ and two parts each induce a subgraph isomorphic to~${K_4}$ that share no triangle and all other parts induce a subgraph isomorphic to~$K_6$.
        \end{itemize}
    \end{enumerate}
\end{lemma}

\begin{proof}
    Since the stitch of two graphs is $3$-connected if and only if both of them are, by the definition of~$\mathcal{C}$ we inductively obtain a tree-decomposition~${\mathcal{T} = (T,\mathcal{V})}$ of~$G$ satisfying properties~\ref{item:td1-1}, \ref{item:td1-2} and~\ref{item:td1-3}, except that if~$\mathcal{C}$ is~$\mathcal{E}_6$, then two distinct parts~$V_x$ and~$V_y$ with~${G[V_x] \cong G[V_y] \cong K_4}$ may share a triangle~$D$. 
    In this case, for any distinct nodes~$s$ and~$t$ on the unique path~$P$ from~$x$ to~$y$ in~$T$, we have~${D = G[V_s \cap V_t]}$ by property~\ref{item:td1-2} and~\ref{item:T3}. 
    We now define a tree~$T'$ from~$T$ by first deleting the edge in~$E(P)$ incident to~$x$, adding a new edge between~$x$ and~$y$ and contracting that edge to a new node~$z$. 
    Moreover, we define~${V'_t := V_t}$ for all~${t \in V(T') \setminus \{z\}}$ and~${V'_z := V_x \cup V_y}$. 
    Now~${G[V'_z]}$ is a stitch of two copies of~$K_4$, and hence is isomorphic to a graph in~${\mathcal{A}(\mathcal{E}_6)}$. 
    Hence, the tree-decomposition~${(T', (V'_t \colon t \in V(T')))}$ satisfies all desired properties. 
\end{proof}

The following lemma now completes the picture, allowing us to determine when a graph with a tree-decomposition as in Proposition~\ref{lem:td1} is in the given class~$\mathcal{C}$. 

\begin{lemma}
    \label{lem:dagger}
    Let $G$ be a $3$-connected graph with tree-decomposition~${(T,\mathcal{V})}$ as in Lemma~\ref{lem:td1}.
    Let~$\mathcal{D}$ be a set of triangles in~$G$ and let~${F \subseteq E(G)}$. 
    The following statements are equivalent.
    \begin{enumerate}
        [label=(\arabic*)]
        \item\label{item:treecomp1} There is a $1$-drawing~$\phi$ of~$G$ with respect to which all triangles in~$\mathcal{D}$ are facial and no edge in~$F$ is crossed.
        \item\label{item:treecomp2} 
        \begin{itemize}
            \item $\mathcal{D}$ and~$\mathcal{H}$ are disjoint;
            \item for each~${D \in \mathcal{H}}$, there is a unique edge~${st \in E(T)}$ such that~${G[V_s \cap V_t] = D}$; 
            \item for every node~$t \in V(T)$ there is a $1$-drawing of~${G[V_t]}$ with respect to which all triangles~${D \subseteq G[V_t]}$ in~${\mathcal{D} \cup \mathcal{H}}$ are facial and all edges in~${E(G[V_t]) \cap F}$ are uncrossed.  
        \end{itemize}
    \end{enumerate}
\end{lemma}

\begin{proof}
    Suppose that~\ref{item:treecomp1} holds. 
    We may assume that~$\phi$ is quasi-rich, since redrawing a crossed edge so that it is no longer crossed can only increase the set of facial triangles. 
    In particular, we may assume that~$\phi\restricted G[V_t]$ is quasi-rich for every node~$t$ of~$T$. 
    Consider a triangle~$D$ in $\mathcal{H}$. 
    Note that by property~\ref{item:td1-3} in Lemma~\ref{lem:td1}, at most one part of~${(T,\mathcal{V})}$ induces a graph which contains~$D$ and is not isomorphic to a graph in ${\{K_6,K_3+C_4\}}$. 
    Let~$k$ be the number of distinct parts of~${(T,\mathcal{V})}$ containing~${V(D)}$, and let ${\{H_i \colon i \in [k]\}}$ be the graphs induced on these parts, so that every graph in ${\{H_i \colon i \in [k-1]\}}$ is isomorphic to a graph in~${\{K_6, K_3+C_4\}}$.
    Note that $k\geq 2$ by the definition of $\mathcal{H}$.

    Let~${x \in [k-1]}$ and~${y \in [k]\setminus \{x\}}$ and let~${H := H_x \cup H_y}$. 
    By Corollary~\ref{cor:drawing-properties}\ref{item:dp1} applied to~${\phi \restricted H_x}$ no edge of~$H_y$ crosses any edge of~$H_x$, since~${\phi\restricted H}$ is quasi-rich.
    It follows that for every component~$C$ of~${H_y - V(D)}$ there is a unique region~$R$ of~${\mathbb{S}^2 \setminus \phi(H_x)}$ containing~${\phi(C)}$. 
    Since~$H_y$ is $3$-connected, the boundary of~$R$ contains~${V(D)}$. 
    The intersection of~${\phi(H_x-E(D))}$ with the boundary of~$R$ is~${V(D)}$. 
    Since~$H_x$ is $3$-connected, the boundary of~$R$ is a simple closed curve. 
    It follows that we can redraw the edges of~$D$ inside~$R$, so that they do not cross any edge of~${H}$, forming a new quasi-rich drawing~$\phi'$ of~${H}$.
    Again by Corollary~\ref{cor:drawing-properties}\ref{item:dp1}, no edge which is crossed with respect to~${\phi\restricted H_x}$ can be redrawn so that it is uncrossed. 
    In particular, the edges of~$D$ are uncrossed with respect to~${\phi\restricted (H)}$. 
    Note that for each edge~${vw \in E(D)}$, at most one of the regions of ${\mathbb{S}^2 \setminus (\phi(vw) \cup \phi'(vw))}$ contains a point of~${\phi(H)}$, since~$H$ is $3$-connected. 
    It follows that~$\phi'$ and~${\phi\restricted H}$ are equivalent. 
    Hence~$\phi(D)$ is the boundary of the region of~${\mathbb{S}^2 \setminus \phi(H_x)}$ containing~$\phi(C)$, and by symmetry is the boundary of every region of~${\mathbb{S}^2 \setminus \phi(H_x)}$ containing~${\phi(C')}$ for any component~$C'$ of~${H_y-V(D)}$. 
    Since~$H_x$ is $3$-connected, it follows that one region of~${\mathbb{S}^2 \setminus \phi(D)}$ contains~${\phi(H_y-V(D))}$, the other region of~${\mathbb{S}^2 \setminus \phi(D)}$ contains~${\phi(H_x-V(D))}$, and~$D$ is not facial with respect to~$\phi$. 
    In particular, ${D \notin \mathcal{D}}$.
    Since no edge of~$H_y$ crosses any edge of~$H_x$ with respect to~$\phi$, it follows that~$D$ is a facial triangle with respect to both~${\phi\restricted H_x}$ and~${\phi\restricted H_y}$, and hence~$D$ is facial with respect to the restriction of~$\phi$ to any subgraph of~$G$ induced on a part of~${(T,\mathcal{V})}$ which contains~${V(D)}$. If $k\geq 3$, then by the pigeonhole principle we can find distinct $i$ and $j$ in $[k]$ such that $\phi(H_{i}-V(D))$ and $\phi(H_{j}-V(D))$ are contained in the same region of $\mathbb{S}^2\setminus \phi(D)$, a contradiction. Hence there is a unique edge $st\in E(T)$ with $D=G[V_s\cap V_t]$. 
    Uncrossed edges with respect to $\phi$ are still uncrossed in all restrictions $\phi$ to subgraphs of $G$, so~\ref{item:treecomp1} implies~\ref{item:treecomp2}.
    
    It is an easy consequence of Lemma~\ref{lem:stitching} that~\ref{item:treecomp2} implies~\ref{item:treecomp1}.
\end{proof}

\begin{corollary}
    \label{cor:td1}
    Each tree-decomposition as in Lemma~\ref{lem:td1} also satisfies 
    \begin{enumerate}
        [resume*=tdenum]
        \item \label{item:td1-5} for each~${D \in \mathcal{H}}$, there is a unique edge~${st \in E(T)}$ such that~${G[V_s \cap V_t] = D}$;
        \item \label{item:td1-4} for each node~${t \in V(T)}$, the graph~${G[V_t]}$ has a $1$-drawing with respect to which all triangles $D\subseteq G[V_t]$ in $\mathcal{H}$ are facial. 
         
    \end{enumerate}
\end{corollary}

\begin{proof}
    By property~\ref{item:td1-3} of Lemma~\ref{lem:td1}, for each edge~${st \in E(T)}$ either~${G[V_s]}$ or~${G[V_t]}$ is isomorphic to a graph in~${\{K_6, K_3 + C_4\}}$. 
    The result now follows from Lemma~\ref{lem:dagger}. 
\end{proof}

Now for the classes~$\mathcal{E}_3$, $\mathcal{E}_4$ and~$\mathcal{E}$, these results immediately yield a characterisation in terms of these tree-decompositions. 

\begin{proposition}
    \label{prop:td2}
    Let~$\mathcal{C}$ be one of $\mathcal{E}_3$, $\mathcal{E}_4$ or~$\mathcal{E}$. 
    A graph~$G$ with at least $9$ vertices is in $\mathcal{C}$ if and only if it has a tree-decomposition~${(T,\mathcal{V})}$ with the following properties.
    \begin{enumerate}
        [label=(\alph*)]
        \item \label{item:td2-1} For each~${t \in V(T)}$ the graph~$G[V_t]$ is isomorphic to a graph in~$\mathcal{A}(\mathcal{C})$; 
        \item \label{item:td2-2} each graph in~${\mathcal{H} := \{ G[V_s \cap V_t] \colon st \in E(T) \}}$ is a triangle; 
        \item \label{item:td2-3} the number of nodes~${t \in V(T)}$ such that $G[V_t]$ is not isomorphic to $K_6$ is at most two, and this inequality is strict 
        unless~$\mathcal{C}$ is~$\mathcal{E}_3$ and two distinct parts each induce a subgraph isomorphic to~${K_3 + C_4}$; 
        \item \label{item:td2-5} for each~${D \in \mathcal{H}}$, there is a unique edge~${st \in E(T)}$ such that~${G[V_s \cap V_t] = D}$; 
        \item \label{item:td2-4} for each node~${t \in V(T)}$, the graph~${G[V_t]}$ has a $1$-drawing with respect to which all triangles~${D \subseteq G[V_t]}$ in~$\mathcal{H}$ are facial. 
    \end{enumerate}
\end{proposition}

\begin{proof}
    First note that each~${G \in \mathcal{C}}$ with at least~$9$ vertices is $3$-connected. 
    Hence the existence of such a tree-decomposition is given from Lemma~\ref{lem:td1} and Corollary~\ref{cor:td1}.
    
    For the other direction, if~$G$ has a tree-decomposition~${(T, \mathcal{V})}$ with these properties, then it follows easily from Lemma~\ref{lem:stitching} that~$G$ is $1$-planar. 
    Using property~\ref{item:td2-3}, a simple induction on~$\abs{V(T)}$ lets us conclude that~$G$ is in~$\mathcal{C}$. 
\end{proof}

Note that not all sufficiently large graphs in~$\mathcal{E}_5$ and~$\mathcal{E}_6$ are $3$-connected, so a structural characterisation in terms of tree-decompositions needs to take separations of order less than~$3$ into account. 
Generalising the well-known notion of the block-cutvertex tree, Tutte~\cite{tutte} proved a decomposition theorem about ``$3$-blocks'' in graphs, for which we will use the following special case. 

\begin{theorem}
    \label{thm:tutte}
    \cite{tutte}
    Let~$G$ be a $2$-connected graph in which every non-trivial $2$-separator induces a clique. 
    Then~$G$ has a tree-decomposition of adhesion~$2$ in which every part induces a $3$-connected graph or a triangle.
\end{theorem}

\begin{corollary}
    \label{cor:2separationedges}
    Let~$G$ be an edge-maximal $1$-planar graph, and let~$\phi$ be a rich $1$-drawing of~$G$. 
    Then~$G$ has a tree-decomposition in which each part induces a triangle or a $3$-connected subgraph, and each adhesion set induces a subgraph isomorphic to~$K_2$, whose unique edge is uncrossed with respect to~$\phi$.
\end{corollary}

\begin{proof}
    Let~$(A,B)$ be a non-trivial separation of order at most~$2$ and 
    let~$R$ be a region of~$\mathbb{S}^2 \setminus \phi(G[A])$ containing a vertex in~${B \setminus A}$. 
    We may assume that no vertex of~${A \setminus B}$ is on the boundary of~$R$, or else we could draw a new uncrossed edge from~${A \setminus B}$ to~${B \setminus A}$ inside~$R$, contradicting the edge-maximality of~$G$. 
    Now the boundary of~$R$ contains a simple closed curve but contains at most two vertices of~${G[A]}$, so in particular it contains a crossing of~${\phi\restricted G[A]}$. 
    It follows that~${A \cap B}$ consists of exactly two vertices~$v$ and~$w$, and that some edge incident to~$v$ crosses some edge incident to~$w$. 
    Since~$\phi$ is rich, ${G[A \cap B] \cong K_2}$, and the edge in~${G[A \cap B]}$ is uncrossed with respect to~$\phi$. 
    The result now follows from Theorem~\ref{thm:tutte}. 
\end{proof}

This allows us to complete our characterisation for the classes~$\mathcal{E}_5$ and~$\mathcal{E}_6$. 
Note that for the sake of simplicity the following characterisation only characterises the graphs in the respective class up to possible deletion of a small constant number of edges that are not contained in any copy of one of the respective cliques. 

\begin{proposition}
    \label{prop:td3}
    A graph~$G$ with ${n \geq 3}$ vertices is in~$\mathcal{E}_6$ if and only if it is a spanning subgraph of some graph~$G'$ with $\mathcal{N}(G',K_6)=\mathcal{N}(G,K_6)$ which has a tree-decomposition ${(T,\mathcal{V})}$ satisfying the following properties. 
    \begin{enumerate}
        [label=(\alph*)]
        \item\label{item:td3-1} For each node~${t \in V(T)}$, the graph $G'[V_t]$ is isomorphic to a graph in~${\mathcal{A}(\mathcal{E}_6)}$; 
        \item\label{item:td3-2} each graph in~$\mathcal{H} := \{ G[V_s \cap V_t] \colon st \in E(T) \}$ is a clique of order~$2$ or~$3$; 
        \item\label{item:td3-4}
        there are exactly $\lfloor (n-3)/3\rfloor$ nodes $t\in V(T)$ such that $G'[V_t]$ is isomorphic to $K_6$;
        \item\label{item:td3-3} for each triangle~${D \in \mathcal{H}}$, there is a unique edge~${st \in E(T)}$ such that~${G'[V_s \cap V_t] = D}$; 
        \item\label{item:td3-5} for each node~${t \in V(T)}$, the graph~${G'[V_t]}$ has a $1$-drawing with respect to which all triangles~${D \subseteq G[V_t]}$ in~$\mathcal{H}$ are facial 
        and all edges of~${G'[V_t] \cap \bigcup \mathcal{H}}$ are uncrossed. 
    \end{enumerate}
\end{proposition}

\begin{proof}
    If there is a supergraph~$G'$ with such a tree-decomposition, then~$G$ contains ${\lfloor(n-3)/3\rfloor}$ subgraphs isomorphic to~$K_6$ by~\ref{item:td3-4}. 
    Hence, by Theorem~\ref{thm:1planarbigcliques} it is sufficient to show that~$G'$ is $1$-planar. 
    This follows easily from Lemma~\ref{lem:stitching}.
    
    For the other direction, suppose $G$ is in $\mathcal{E}_6$, and let~$G'$ be an $n$-vertex edge-maximal $1$-planar graph containing~$G$. 
    By Theorem~\ref{thm:1planarbigcliques}, $G'$ is also in~$\mathcal{E}_6$. 
    Let~${\mathcal{T}' := (T', (V'_t \colon t \in V(T)))}$ be a tree-decomposition of~$G'$ as in Corollary~\ref{cor:2separationedges}.
    It is easy to construct a $1$-planar graph~$G''$ with a tree-decomposition ${(T',(V''_t \colon V(T)))}$ such that each adhesion set has size two, for each~${t \in V(T)}$ the graph~${G[V''_t]}$ is in~$\mathcal{E}_6$ and~${\abs{V''_t} = \abs{V'_t}}$. 
    Hence, by Theorem~\ref{thm:1planarbigcliques}, each part of $\mathcal{T}'$ induces a graph in $\mathcal{E}_6$. 
    
    By Lemmas~\ref{lem:td1} and~\ref{lem:dagger}, for each~${t \in V(T')}$, the graph $G[V_t]$ has a tree-decomposition ${\mathcal{T}^t := (T^t, (V^t_s))}$ satisfying properties~\ref{item:td3-1}, \ref{item:td3-2}, \ref{item:td3-3} as well as 
    \begin{enumerate}
        [label=($e'$)]
        \item for each node~${s \in V(T^t)}$ the graph~$G'[V^t_s]$ has a $1$-drawing with respect to which all triangles~${D \subseteq G'[V^t_s]}$ in~$\mathcal{H}$ are facial and all edges of~${G'[V^t_s] \cap \bigcup \mathcal{H}'}$ 
        are uncrossed, 
        where~$\mathcal{H}' := \{ G'[V_s' \cap V_t'] \colon st \in E(T') \}$. 
    \end{enumerate}
    
    We now construct a tree~$T$ as the disjoint union of the trees $T^t$ for all~${t \in V(T')}$ by adding for each edge~$st$ of~$T'$ an edge between some arbitrary nodes~${u \in V(T^s)}$ and~${v \in V(T^t)}$ for which~${V_u^s \cap V_v^t = V'_s \cap V'_t}$. 
    Setting for each~${s \in V(T)}$ the part~$V_s$ to be equal to the part~$V_s^t$ for the unique~${t \in V(T')}$ with~${s \in V(T^t)}$ yields a tree-decomposition that is easily seen to satisfy properties~\ref{item:td3-1}, \ref{item:td3-2}, \ref{item:td3-3} and~\ref{item:td3-5}. 
    By Theorem~\ref{thm:1planarbigcliques}, it also satisfies~\ref{item:td3-4}.
\end{proof}

For the class~$\mathcal{E}_5$ we observe the following.

\begin{proposition}
    \label{prop:E5}
    An $n$-vertex graph~$G$ is in~$\mathcal{E}_5$ if and only if~${G \in \mathcal{E}_6}$ and either~$n$ is congruent to~$1$ modulo~$3$ or~${G \in \mathcal{E}}$. 
\end{proposition}

\subsection{Efficiently generating the extremal graphs}
\label{subsec:efficent}

We begin by observing that the $3k$-vertex extremal graphs for any integer~${k \geq 2}$ are isomorphic to the strong product of a triangle and a path of length~$k$, as mentioned in the introduction. 

\begin{corollary}
    \label{cor:k6stitching}
    Let~$\mathcal{C}$ be one of $\mathcal{E}_3$, $\mathcal{E}_4$, $\mathcal{E}_5$, $\mathcal{E}_6$ or~$\mathcal{E}$ and let~${k \geq 2}$ be an integer, then a $3k$-vertex graph~$G$ is in~$\mathcal{C}$ if and only if~$G$ is isomorphic to the strong product of a triangle and a path of length~${k-1}$. 
\end{corollary}

\begin{proof}
    By the characterisations given in Propositions~\ref{prop:td2}, \ref{prop:td3} and~\ref{prop:E5}, 
    a $3k$-vertex graph~$G$ in~$\mathcal{C}$ has a tree-decomposition~${(T,\mathcal{V})}$ in which each part is isomorphic to~$K_6$ with a drawing in which each adhesion set is a distinct facial triangle. 
    By Corollary~\ref{cor:drawing-properties}\ref{item:dp3}, each part has at most two such triangles, and these triangles are vertex disjoint. 
    Hence~$T$ is a path of length~${k-2}$, 
    and the result now follows from the definition of the strong product. 
\end{proof}

If~$n$ is not divisible by~$3$, Propositions~\ref{prop:td2}, \ref{prop:td3} and~\ref{prop:E5} still allow us to generate all extremal graphs of size~$n$ in polynomial time. 

If~$\mathcal{C}$ is~$\mathcal{E}_3$, we observe that~${K_3 + C_4}$ has a unique $1$-drawing which has exactly two facial triangles (see Lemma~\ref{lem:unique1drawing} and Corollary~\ref{cor:drawing-properties}\ref{item:dp3}). 
Moreover, by Lemma~\ref{lem:rich1drawing}, each $1$-drawing of~${K_2 + \overline{P_6}}$ is rich, and there is a unique rich $1$-drawing of~${K_2 + \overline{P_6}}$ up to weak equivalence by the computer search in Lemma~\ref{lem:computation8}. 
Hence, the tree~$T$ in the tree-decomposition from Proposition~\ref{prop:td2} is again a path of length~${\lfloor n/3 \rfloor}$. 
With the restrictions on the parts given in the characterisation in Proposition~\ref{prop:td2}, we conclude that we can generate all $n$-vertex graphs in~$\mathcal{E}_3$ in polynomial time. 

If~$\mathcal{C}$ is either~$\mathcal{E}_4$ or~$\mathcal{E}$, then as there is at most one part inducing a graph not isomorphic to~$K_6$, we observe that~$T$ is the subdivision of a star. 
Moreover,~$K_4$ contains only four triangles, and 
since every simple $1$-drawing of~$K_5$ is rich by Lemma~\ref{lem:rich1drawing}, and it is easy to observe that the drawing depicted in Figure~\ref{fig:k5} is the unique $1$-drawing of $K_5$ up to weak equivalence, 
we conclude that~$T$ has maximum degree at most~$4$. 
As before, the restrictions on the parts given in the characterisation in Proposition~\ref{prop:td2} allow us to generate all $n$-vertex graphs in~$\mathcal{E}_4$ or~$\mathcal{E}$ in polynomial time. 

In the case of~$\mathcal{E}_6$, note that the only non-planar graphs in~${\mathcal{A}(\mathcal{E}_6)}$ are~$K_6$ and~$K_5$. 
In a planar drawing of a graph, every non-separating triangle is facial. 
Thus in order to apply Proposition~\ref{prop:td3}, the only remaining graph to consider is~${K_3+\overline{K_2}}$, since this is the only graph in~${\mathcal{A}(\mathcal{E}_6)}$ with a separating triangle. 
It is easy to show that there is only one $1$-drawing up to weak equivalence such that the separating triangle is facial, see Figure~\ref{fig:k5}. 
Thus we can again generate all $n$-vertex graphs in~$\mathcal{E}_6$ in polynomial time.

By Proposition~\ref{prop:E5}, we can extract all $n$-vertex graphs in~$\mathcal{E}_5$ from the previously generated lists.

    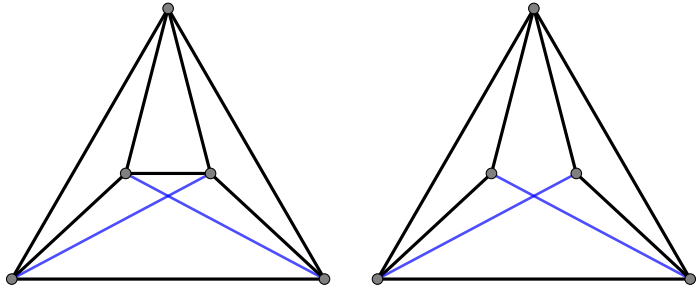
\begin{figure}[htbp]
        \centering
        \begin{tikzpicture}
            [scale=0.8]
            \tikzset{vertex/.style = {circle, draw, fill=black!50, inner sep=0pt, minimum width=4pt}}
            \tikzset{edge0/.style = {line width=1.2pt, black}}
            \tikzset{edge1/.style = {line width=1pt, blue, opacity=0.7}}
            
            \coordinate (v0) at (210:3) {};
            \coordinate (v1) at (330:3) {};
            \coordinate (v2) at (90:3) {};
            \coordinate (w0) at (160:0.75) {};
            \coordinate (w1) at (20:0.75) {};
            
            \foreach \i in {0,1,2} {
                \pgfmathtruncatemacro{\j}{Mod(\i+1,3)}
                \draw [edge0] (v\i) edge (v\j) {};
                \foreach \k in {0,1} {
                \ifthenelse{\i = 2 \OR \i = \k}{
                    \draw [edge0] (v\i) edge (w\k) {};
                }{
                    \draw [edge1] (v\i) edge (w\k) {};
                }
                };
            };
            \draw [edge0] (w0) edge (w1) {};
            
            \foreach \v in {v0,v1,v2,w0,w1} {
                \node [vertex] at (\v) {};
            };
        \end{tikzpicture}
        \quad
        \begin{tikzpicture}
            [scale=0.8]
            \tikzset{vertex/.style = {circle, draw, fill=black!50, inner sep=0pt, minimum width=4pt}}
            \tikzset{edge0/.style = {line width=1.2pt, black}}
            \tikzset{edge1/.style = {line width=1pt, blue, opacity=0.7}}
            
            \coordinate (v0) at (210:3) {};
            \coordinate (v1) at (330:3) {};
            \coordinate (v2) at (90:3) {};
            \coordinate (w0) at (160:0.75) {};
            \coordinate (w1) at (20:0.75) {};
            
            \foreach \i in {0,1,2} {
                \pgfmathtruncatemacro{\j}{Mod(\i+1,3)}
                \draw [edge0] (v\i) edge (v\j) {};
                \foreach \k in {0,1} {
                \ifthenelse{\i = 2 \OR \i = \k}{
                    \draw [edge0] (v\i) edge (w\k) {};
                }{
                    \draw [edge1] (v\i) edge (w\k) {};
                }
                };
            };
            
            \foreach \v in {v0,v1,v2,w0,w1} {
                \node [vertex] at (\v) {};
            };
        \end{tikzpicture}
        \caption{On the left, the unique simple $1$-drawing of~$K_5$ (up to weak equivalence), and on the right, the only simple $1$-drawing of~$K_3 + \overline{K_2}$ (up to weak equivalence) such that the separating triangle is facial. }
        \label{fig:k5}
    \end{figure}

\section{Concluding remarks}\label{sec:conclusion}

For $1$-planar graphs, our results show that for any~$t$ and~$t'$ in~${\{3,4,5,6\}}$, there is not much difference between maximising cliques of size~$t$ and maximising cliques of size~$t'$. 
In fact, when~$n$ is divisible by~$3$, the $n$-vertex graphs with ${f_t(n)}$ copies of~$K_t$ also have ${f_{t'}(n)}$ copies of~$K_{t'}$. 
This behaviour is similar to what we observe for planar graphs, where Apollonian networks are the graphs maximising both triangles and cliques of size~$4$. 
Interestingly however, the $1$-planar graphs which maximise cliques of any size greater than~$2$ do not maximise the number of edges. The $n$-vertex $1$-planar graphs with $4n-8$ edges, called 
optimal $1$-planar graphs, form an interesting class in their own right. Indeed, much of the research into $1$-planar graphs has focused explicitly on this subclass. 
We therefore pose the following question.

\begin{question}
    Let~$n$ and~$t$ be positive integers with~${t \in \{3,4,5\}}$ and~${n \geq 10}$. 
    What is the maximum number of subgraphs isomorphic to~$K_t$ in an $n$-vertex $1$-planar graph with~${4n-8}$ edges? 
\end{question}

For~${t=3}$, Lemma~\ref{lem:trueplanarskeletons} implies that this question is equivalent to maximising the number of non-trivial $3$-separators in a $3$-connected planar quadrangulation. 
In this case, we suspect the answer is given by following construction. 
Let~$H_0$ be the $4$-cycle ${v_0v_1v_2v_3v_0}$, let~$\phi_0$ be a planar drawing of~$H_0$, and let~$F^0_1$ and~$F^0_2$ be the faces of~$\phi_0$. 
For each positive integer~$i$, let~$H_i$ be obtained from~$H_{i-1}$ by adding a new vertex~$v_{i+3}$ adjacent to~$v_{i+2}$ and the non-neighbour of~$v_{i+2}$ in the boundary of~$F^{i-1}_2$ and let~$\phi_i$ extend~$\phi_{i-1}$ to a drawing of~$H_i$ by drawing the new vertices and edges inside~$F^{i-1}_2$. 
Let~${F^i_1 := F^{i-1}_1}$ and let~$F^i_2$ be a face of~$\phi_i$ whose boundary contains~$v_{i+3}$, (chosen arbitrarily). 
Finally let~$G_i$ be obtained from~$H_i$ by taking two disjoint copies of the cube~$Q_3$, and identifying a $4$-cycle in the first copy with the boundary cycle of~$F^i_1$ and a $4$-cycle in the second copy with the boundary cycle of~$F^i_2$. 
We will call the class of graphs~$G_i$ constructed this way~$\mathcal{G}^\ast$. 
Every graph in~$\mathcal{G}^\ast$ is a $3$-connected planar quadrangulation, and it can be shown that $n$-vertex graphs in~$\mathcal{G}^\ast$ have exactly ${4n-14}$ non-trivial $3$-separators. 
Lemma~\ref{lem:trueplanarskeletons} and Lemma~\ref{lem:K6freebound} together imply that this is within ten of the maximimum possible value. 
We conjecture the following. 

\begin{conjecture}
    For~${n \geq 14}$, the $n$-vertex $3$-connected planar quadrangulations with the maximum number of $3$-separators are exactly the $n$-vertex graphs in~$\mathcal{G}^\ast$.
\end{conjecture}

\medskip

In another direction, it is natural to consider extending our results to $k$-planar graphs for higher values of $k$. For $k\geq 3$, not even the maximum number of edges is known for $n$-vertex $k$-planar graphs, so counting larger cliques is likely to be an extremely difficult problem. However, counting triangles in $2$-planar graphs is a logical next step toward broadening our understanding. Here, we conjecture the following.
\begin{conjecture}\label{conj:2planar}
    For~${n \geq 7}$, the maximum number of triangles in an $n$-vertex $2$-planar graph is at most $(17n-49)/2$. 
\end{conjecture}
The bound in Conjecture~\ref{conj:2planar} is achieved by the $2$-planar graphs formed by stitching copies of~$K_7$ together. 
This stitching is possible since~$K_7$ has a $2$-drawing with two facial triangles, see Figure~\ref{fig:K7}. 
It may even be the case that for all positive integers~$k$, maximising copies of the largest $k$-planar cliques also maximises triangles, up to an additive constant, and vice versa.

    \begin{figure}[htbp]
        \centering
        \begin{tikzpicture}
            [scale=1]
            \tikzset{vertex/.style = {circle, draw, fill=black!50, inner sep=0pt, minimum width=4pt}}
            \tikzset{edge0/.style = {line width=1.2pt, black}}
            \tikzset{edge1/.style = {line width=1pt, blue, opacity=0.7}}
            \tikzset{edge2/.style = {line width=1pt, red, opacity=0.7}}
            
            \coordinate (v1) at (210:3) {};
            \coordinate (v2) at (330:3) {};
            \coordinate (v3) at (90:3) {};
            \coordinate (v4) at (210:1.5) {};
            \coordinate (v5) at (330:1.5) {};
            \coordinate (v6) at (90:1.5) {};
            \coordinate (v7) at (85:2.1) {};
            
            \draw [edge0] (v1) edge (v2) {};
            \draw [edge0] (v2) edge (v3) {};
            \draw [edge0] (v3) edge (v1) {};
            
            \draw [edge0] (v4) edge (v5) {};
            \draw [edge0] (v5) edge (v6) {};
            \draw [edge0] (v6) edge (v4) {};
            
            \draw [edge0] (v1) edge (v4) {};
            \draw [edge0] (v2) edge (v5) {};
            \draw [edge1] (v1) edge (v5) {};
            \draw [edge1] (v2) edge (v4) {};
            
            \draw [edge0] (v7) edge (v3) {};
            \draw [edge0] (v7) edge (v6) {};
            \draw [edge2] (v7) edge (v1) {};
            \draw [edge1] (v7) edge (v2) {};
            \draw [edge2] (v7) edge (v4) {};
            \draw [edge1] (v7) edge (v5) {};
            \draw [edge2] (v1) edge (v6) {};
            \draw [edge2] (v3) edge (v4) {};
            \draw [edge2] (v2) edge (v6) {};
            \draw [edge2] (v3) edge (v5) {};
            
            \draw [edge2] (v3) edge (v6) {};
            
            \foreach \i in {1,...,7} {
                \node [vertex] at (v\i) {};
            };
        \end{tikzpicture}
        \caption{A $2$-planar drawing of~$K_7$. The black edges are the edges of the true-planar skeleton, the blue edges are crossed exactly once and the red edges are crossed exactly twice.}
        \label{fig:K7}
    \end{figure}
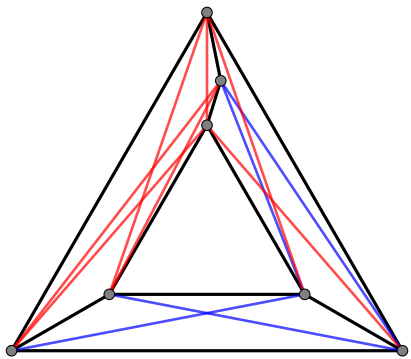

\end{document}